\documentclass[11pt]{article}
\usepackage{amsfonts}
\usepackage{amssymb}
\usepackage{amsthm}
\usepackage{amsmath}
\usepackage{graphicx}
\usepackage{fixltx2e}
\usepackage{gensymb}
\usepackage[numbers]{natbib}

\usepackage{indentfirst}
\usepackage{mathrsfs}
\usepackage{tikz}
\usepackage{graphics}
\usepackage{braket}
\theoremstyle{plain}
\usepackage{enumerate}
\usepackage{upgreek}
\usepackage{siunitx}
\newtheorem{theorem}{\textbf{Theorem}}[section]
\newtheorem{lemma}{\textbf{Lemma}}[section]
\newtheorem{proposition}{\textbf{Proposition}}[section]
\newtheorem{corollary}{\textbf{Corollary}}[section]
\newtheorem{remark}{\textbf{Remark}}[section]
\newtheorem{definition}{\textbf{Definition}}[section]

\allowdisplaybreaks[4]

\def\R3{\mathbb{R}^3} 
\def\R{\mathbb{R}}
\def\F2o{\overline{F_2}}

\def \eps{\varepsilon}

\newcommand{\beq}[0]{\begin{equation}}
\newcommand{\eeq}[0]{\end{equation}}

\newcommand{\ff}[0]{\varphi}
\newcommand{\mint}[0]{\int_\Omega}

\newcommand{\Tint}[0]{\int_0^T}

\newcommand{\dert}[0]{\frac{d}{dt}}

\newcommand{\EE}[0]{{\mathcal E}}

\newcommand{\lep}[0]{\lambda_\eps}
\newcommand{\vc}[1]{\mathbf{#1}}
\newcommand{\mt}[1]{\mathsf{#1}}

\newcommand{\curl}[0]{\nabla\times}

\DeclareMathOperator{\sgn}{sgn}
\newcommand{\meas}[0]{\mathscr{L}^3}

\DeclareMathOperator{\supp}{supp}
\DeclareMathOperator{\tr}{tr}

\DeclareMathOperator{\argmin}{argmin}
\DeclareMathOperator{\cof}{\mt{cof}}

\DeclareMathOperator{\rank}{rank}

\newcommand{\mres}{\mathbin{\vrule height 1.6ex depth 0pt width
0.13ex\vrule height 0.13ex depth 0pt width 1.3ex}}

\usepackage[hang,flushmargin]{footmisc}
\makeatletter
\def\blfootnote{\gdef\@thefnmark{}\@footnotetext}
\makeatother

\begin{document}

\title{Analysis of a moving mask approximation for martensitic transformations
%\thanks{This work was supported by the Engineering and Physical Sciences Research Council [EP/L015811/1]. 
%}
}
%\subtitle{Moving mask approximation}


%\title{Modelling moving interfaces in reversible martensitic transformations}

\author{
{\sc Francesco Della Porta}\\
Mathematical Institute, University of Oxford\\
Oxford OX2 6GG, UK\\
\textit{francesco.dellaporta@maths.ox.ac.uk\footnote{Current institution: Max-Planck Institute for Mathematics in the Sciences, Leipzig, Germany, 
\textit{Francesco.DellaPorta@mis.mpg.de}}}
}

\date{\today}
\maketitle
%for his helpfull feedback in improving the quality of this manuscript
\blfootnote{%\noindent\rule{6.8cm}{0.4pt}\\
\textbf{Acknowledgements:} {This work was supported by the Engineering and Physical Sciences Research Council [EP/L015811/1]. The author would like to thank John Ball for his helpful suggestions and feedback which greatly improved this work, as well as Richard James, Giacomo Canevari and Xian Chen for the useful discussions.}\\
%\textbf{Declarations of interest:} none.
}
%\blfootnote{\noindent\rule{3.8cm}{0.4pt}\\
%\textbf{Compliance with ethical standards}\\ 
%The author declares that he has no conflict of interest.\\
%This article does not contain any studies with human or animal subjects.
%}

\begin{abstract}
In this work we introduce a moving mask approximation to describe the dynamics of austenite to martensite phase transitions at a continuum level. In this framework, we prove a new type of Hadamard jump condition, from which we deduce that the deformation gradient must be of the form $\mt 1 +$ $\vc a\otimes \vc n$ a.e. in the martensite phase. This is useful to better understand the complex microstructures and the formation of curved interfaces between phases in new ultra-low hysteresis alloys such as Zn\textsubscript{45}Au\textsubscript{30}Cu\textsubscript{25}, {and provides a selection mechanism for physically-relevant energy-minimising microstructures. In particular, we use the new type of Hadamard jump condition to deduce a rigidity theorem for the two well problem. The latter provides more insight on the cofactor conditions, particular conditions of supercompatibility between phases believed to influence reversibility of martensitic transformations.}
%Insert your abstract here. Include keywords, PACS and mathematical
%subject classification numbers as needed.
\end{abstract}

\begin{section}{Introduction}
The aim of this work is to study from a mathematical point of view the complex microstructures arising during the austenite to martensite phase transition in ultra-low hysteresis alloys such as Zn\textsubscript{45}Au\textsubscript{30}Cu\textsubscript{25} (see \cite{JamesNew}). 
Austenite to martensite transitions are solid to solid phase transitions, in which the underlying crystal lattice of an alloy experiences a change of shape as temperature is moved across a certain critical temperature $\theta_T$. When the temperature is above $\theta_T$, the alloy has a unique crystalline structure, called austenite, which is energetically preferable; when the temperature is lowered below $\theta_T$, the energetically preferable state for the crystal is no longer austenite but martensite, which usually has more then one variant.
Often, a change of crystalline structure implies a change in the macroscopic properties of the material, which can thus be controlled by changing the temperature of the sample. 
%As it is easy to imagine, alloys with this ability, of which the most famous example are shape memory alloys, have a wide range of applications: from automotive to medical industries, from eco-friendly refrigerators to energy conversion devices. 
A serious obstacle to practical applications of shape-memory and other such materials is reversibility of the transformation. Indeed, after a small number of cycles, one can usually observe a shift in the transition temperature and in the latent heat. Furthermore, the formation of micro-cracks during the phase transition often leads to early failure by rupture.  \\
\indent
An important step towards understanding the factors influencing reversibility can be found in %\cite{JamesMuller} and
\cite{JamesHyst}. There, the authors study particular conditions of geometric compatibility between martensitic variants called cofactor conditions, that were first introduced in \cite{BallJames1}.
Among these conditions, there is the requirement that the middle eigenvalue of the lattice transformation matrices is equal to $1$, which was previously shown (see \cite{JamesMuller}) to influence reversibility. In \cite{JamesHyst}, the authors prove that under the cofactor conditions no elastic transition layer is needed to make simple laminates compatible with austenite,
%interfaces without requiring an interface layer%More surprisingly, no interface layer is needed between austenite and martnsite
%{\color{red}a continuum of exact austenite to martensite interfaces are possible}
and point out that this fact might have important consequences on the reversibility of the phase transitions. 
Indeed, the authors observe that transition layers are intuitively both a cause of thermal hysteresis, and of the formation of dislocations and nucleation of micro-cracks, that, after many cycles, induce loss of good reversibility properties.\\
\indent
The recent fabrication announced in \cite{JamesNew} of Zn\textsubscript{45}Au\textsubscript{30}Cu\textsubscript{25}, the first material closely satisfying the cofactor conditions (the relative error is of order $10^{-4}$), partially confirms this conjecture. Indeed, this material exhibits ultra-low hysteresis and does not seem to incur any loss of reversibility after more than $16,000$ thermal cycles. {%\color{red}
We refer the reader interested in ultra-low hysteresis alloys also to \cite{Chluba}, where the fabrication of a new material undergoing $10^7$ cycles with very little fatigue is announced. A discussion on the relation between the cofactor conditions and ultra-low hysteresis alloys can be found in \cite{JamesSur,JamesSur2}. } \\
\indent
As remarked in \cite{JamesNew}, it is intriguing and unusual that martensitic microstructures in Zn\textsubscript{45}Au\textsubscript{30}Cu\textsubscript{25} are drastically different in consecutive thermally induced transformation cycles, this being partially motivated by the fact that the cofactor conditions are close to being satisfied by both some type I and some type II twins, which can all form zero energy interfaces with austenite. \\
\indent
%\color{red}
The aim of this paper is to further study these microstructures, and to identify a common characterization for all of them. To this end we start from the following observation: from the dynamical point of view, it looks as if in every thermally induced phase transformation cycle in Zn\textsubscript{45}Au\textsubscript{30}Cu\textsubscript{25} there was a mask moving across the domain covering and uncovering martensite microstructures. This is equivalent to saying that the martensitic microstructures do not change after the phase transition has happened, which seems to be a particularly legitimate approximation in materials satisfying the cofactor conditions, where no interface layer is needed between phases. But this hypothesis makes sense also in many other materials as long as one considers macroscopic deformations. \\
\indent
{As a first step, we give a mathematical characterisation of the moving mask approximation and we frame it in the context of nonlinear elasticity, where phase changes are interpreted as elastic deformations. In particular, microstructures satisfying the moving mask approximation are special solutions to a simplified model for the dynamics of martensitic transformations which was introduced in \cite{FDP1}. The model in \cite{FDP1} is derived from the equations for the conservation of energy and momentum in the context of dynamics for nonlinear elasticity, and describes the evolution of the phase interface as a moving shock wave (see Section \ref{Sep phase})}.
%
%. In this context, one can describe the evolution of the phase 
%interface as a moving shock wave.
%The model equations have to be closed with some constitutive relation for the phase boundary velocity. This should depend on the local temperature of the sample, the thermal boundary conditions, and on the local martensitic microstructures, and is therefore not easy to determine.} 
%Nonetheless, by deriving a new type of Hadamard jump condition which is based on the moving mask assumption, we prove that every martensitic microstructure arising from the evolution and satisfying this assumption must be of the form
Then, by deriving a new type of Hadamard jump condition (Section \ref{Hadamard generalizations}), we prove that every martensitic microstructure satisfying  the moving mask assumption and some further technical hypotheses must be of the form
\beq 
\label{Rango 1}\nabla \vc y(\vc x) = \mt 1+ \vc a( \vc x)\otimes \vc n( \vc x),\qquad \text{a.e.},
\eeq 
where $\vc n(\vc x)$ is, up to a change of sign, the phase interface normal at $\vc x$ when the point $\vc x$ is on the interface. The above result does not follow directly from the assumption that the deformation gradient is unchanged after the phase transition, and is not a direct consequence of previously known Hadamard jump conditions if $\vc y$ is just Lipschitz as in our case. 
We refer the interested reader to Section \ref{nonlonear elastic} for a brief review on known versions of the Hadamard jump conditions and on why they do not imply \eqref{Rango 1}. {%\color{red}
Subsequent experiments (see \cite{XC}) have measured $\|\cof(\nabla \vc y-\mt 1)\|$ in a sample of Zn\textsubscript{45}Au\textsubscript{30}Cu\textsubscript{25}. The measured values are of the order of $10^{-4}$, which seem to be small enough to partially confirm the validity of \eqref{Rango 1} and of the moving mask approximation.} {Conversely, from the moving mask hypotheses and our new Hadamard jump condition, we can reconstruct the position of the austenite-martensite interface during the phase transition from a martensitic deformation gradient.}
\\
\indent Under suitable hypotheses, we prove also that $\nabla \cdot \vc a =0$. This result relies on a Hadamard jump condition for strains in $BV(\Omega)$ that is proved in Section \ref{Basic} and generalizes that in \cite{MullerDolzman}. As a consequence of the fact that $\nabla\vc y = \mt 1+ \vc a\otimes \vc n$ almost everywhere in the martensitic microstructures, we can prove a rigidity theorem for compound twins and a result that allows for a better understanding of the nature of curved austenite-martensite interfaces for type I twins. {Under some further assumption, we can extend the rigidity result to a general two well problem not satisfying the cofactor conditions. This result explains the importance of satisfying the cofactor conditions in order to have non-constant average deformation gradients of the form \eqref{Rango 1}, which are obtained by finely mixing two martensitic variants.}\\% and the related riverine microstructures (see \cite{JamesNew}).\\
\indent
The dynamics of the phase transition in Zn\textsubscript{45}Au\textsubscript{30}Cu\textsubscript{25} are very complex, and far from being completely understood. As in the static case, a major obstacle towards a good understanding of the phenomenon remains the lack of a characterization of the quasiconvex hull of the set of possible deformation gradients. Nonetheless, our results provide an interesting set of tools that can be used to understand further the complex microstructures arising in martensitic phase transitions. {Indeed, our moving mask hypothesis can be seen as a selection mechanism for physicaly relevant energy minimising microstructures arising in thermally induced martensitic transformations.}%Furthermore, our model can be easily generalized to a wider class of materials, other than Zn\textsubscript{45}Au\textsubscript{30}Cu\textsubscript{25}, which is the main object of our work.  
\\
\indent
Further investigation on why martensitic microstructures are so different in different thermal cycles in Zn\textsubscript{45}Au\textsubscript{30}Cu\textsubscript{25} is carried out in \cite{FDP3}. Indeed, in \cite{FDP3} we show that this material satisfies some further conditions of compatibility (on top of the cofactor conditions) that makes the set of possible macroscopic deformation gradients of the form \eqref{Rango 1} unusually large. %{\color{red}This might explain why we observe a different microstructure at every transformation cicle, despite the constraint \eqref{Rango 1}. }
\\
\indent
The plan for the paper is the following: in Section \ref{nonlonear elastic} we give a brief overview of the nonlinear elasticity model, and introduce concepts which will be useful for our analysis, namely twinning, the cofactor conditions and k-rectifiable sets. In Section \ref{Sep phase} we recall results from \cite{FDP1} and in this context we introduce the mathematical definition of moving mask approximation. 
In Section \ref{Hadamard generalizations} we prove a dynamic variant of the Hadamard jump condition for Lipschitz functions and curved interfaces. As explained above, the results rely on the hypothesis that the deformation gradient remains constant in time at a point of the domain, once the phase transition has occurred. %In Section \ref{PDE theory} we study analytically the evolution in a simple 1-dimensional example. 
{The last two sections are devoted to proving the rigidity results, and some results on moving austenite-martensite interfaces and on possible microstructures that can be explained using our model.}
\end{section}

%%%%%%%%%%%%%%%%%%%%%%%%%%%%%%%%%%%%%%%%%
%%%%   NONLINEAR ELASTICITY MODEL
%%%%%%%%%%%%%%%%%%%%%%%%%%%%%%%%%%%%%%%%%
\section{Preliminaries}

\begin{subsection}{Nonlinear elasticity model}
\label{nonlonear elastic}
In order to describe austenite to martensite phase transitions in crystalline solids, one of the most successful mathematical continuum models is nonlinear elasticity, which has proved to capture many aspects of the physical phenomena such as the formation of twins (see \cite{BallJames1}) and to be useful in understanding related behaviour such as the shape-memory effect (see \cite{BattSme}), and, more recently, hysteresis (see \cite{JamesMuller}). In this and the next subsection, we give a brief overview of the theory following closely \cite{BallKoumatos,JamesHyst}. For more details we refer the reader to \cite{BallJames1,BallJames2,Batt}. \\% to whom I refer the reader for more details.
\indent
The nonlinear elasticity model is based on the idea of looking at changes in the crystal lattice as elastic deformations in the continuum mechanics framework. Following \cite{BallJames1}, we hence assume that the deformations minimize a free energy of the type
\beq
\label{energia}
\EE(\vc y,\theta)  = \mint\phi(\nabla \vc y(\vc x),\theta)\,\mathrm d\vc x.
\eeq
Here, $\theta$ denotes the temperature of the crystal. Three different regimes are distinguishable depending on this parameter: $\theta<\theta_T$ and $\theta>\theta_T$, where respectively martensite and austenite phases minimize the energy, and $\theta=\theta_T$ where these are energetically equivalent. In~\eqref{energia}, the bounded Lipschitz domain (open and connected) $\Omega$ stands for the reference configuration of undistorted austenite at $\theta=\theta_T$ and $\vc y(\vc x)$ denotes the position of the particle $\vc x\in\Omega$ after the deformation. Finally, $\phi$ is the free-energy density, depending on the temperature $\theta$ and the deformation gradient $\nabla \vc y$, satisfying the following properties:
\begin{itemize}
\item $\mathcal D:=\{\mt F\in\mathbb{R}^{3\times 3}\colon \det \mt F >0\}$, %$$\phi:\mathcal D\times \mathbb{R}_+\mapsto \mathbb{R}$$ is a 
$$\phi(\cdot,\theta)\colon\mathcal D\to \mathbb{R}$$ is a function bounded below by a constant depending on $\theta$ for each $\theta>0$;
% continuous on $\mathcal D$ for every $\theta>0$;
\item $\phi(\cdot,\theta)$ satisfies frame-indifference, i.e., for all $\mt F\in\mathcal D$ and all rotations $\mt R\in SO(3)$, $\phi(\mt R \mt F,\theta)=\phi(\mt F,\theta)$. This property reflects the invariance of the free-energy density under rotations;
\item $\phi$ has cubic symmetry, i.e., $\phi(\mt {FQ},\theta)=\phi(\mt F,\theta)$ for all $\mt F\in\mathcal D$ and all rotations $\mt Q$ in the symmetry group of austenite $\mathcal{P}^{24}$, the group of rotations sending a cube into itself (see \cite{Batt} for more details);
\item denoting by $K_\theta$ the set of minima for the free-energy density at temperature $\theta$, i.e., $K_\theta:=\{\mt F\in\mathcal D\colon\,\mt F\in\argmin(\phi(G,\theta))\}$,
\beq
K_\theta = 
\begin{cases}
\alpha(\theta)SO(3),\qquad &\theta>\theta_T\\ 
SO(3)\cup\bigcup_{i=1}^NSO(3)\mt U_i(\theta_T), \qquad &\theta=\theta_T\\
\bigcup_{i=1}^NSO(3)\mt U_i(\theta), \qquad &\theta<\theta_T .
\end{cases}
\eeq
Here, $\alpha(\theta)$ is a scalar dilatation coefficient satisfying $\alpha(\theta_T)=1$, while {%\color{red}
$\mt U_i(\theta)\in\R^{3\times3}_{Sym^+}$ are the $N$ positive definite symmetric matrices corresponding to the transformation from austenite to the $N$ variants of martensite at temperature $\theta$. Here and below $\R^{3\times3}_{Sym^+}$ represents the set of $3\times3$ symmetric and positive definite matrices.} From now on, we omit the dependence on the temperature in $K_\theta$ when $\theta<\theta_T$, and neglect the dependence on $\theta$ of the $\mt U_i$. We remark that for each $\mt U_i,\mt U_j$ there exists $\mt R\in \mathcal{P}^{24}$ such that $\mt R^T\mt U_j\mt R = \mt U_i$, so that $\mt U_i,\mt U_j$ share the same eigenvalues.%;
%\item interfacial energy contributions are ignored.
\end{itemize}

Based on both experimental evidence and the mathematical complexity of other cases, most results in the literature are related to planar austenite-martensite interfaces $\{\vc x\in \R^3\colon\,\vc x\cdot \vc n = k\}$, with normal $\vc n$, at $\theta=\theta_T$. In this case, under suitable conditions on the lattice deformation and for some $\vc n\in\R^3$, it is possible to construct a sequence $\vc y^j$ such that
\begin{align}
\label{convergenza in misura}
&\nabla \vc y^j\to SO(3)\qquad &\text{in measure}\qquad &\text{for }\vc x\cdot \vc n<k\\
\label{convergenza in misura 2}
&\nabla \vc y^j\to \bigcup_{i=1}^N SO(3)\mt U_i\qquad &\text{in measure}\qquad &\text{for }\vc x\cdot \vc n>k.
\end{align}
Denoting by $\meas$ the three-dimensional Lebesgue measure, we notice that~\eqref{convergenza in misura}-\eqref{convergenza in misura 2} imply
$$
\lim_{j \to\infty}\meas\{\vc x\in\Omega\colon \nabla \vc y^j(\vc x)\notin K_{\theta_T}\}=0,
$$
and, under some further hypotheses on $\phi$, $\vc y^j$ is a minimizing sequence for $\EE(\cdot,\theta_T)$ (see \cite{BallJames1} for more details). Furthermore, since $\vc y^j$ can be constructed so as to be bounded in $W^{1,\infty}(\Omega,\R^3)$, there exists a subsequence $\vc y^{k_j}$ and $\vc y\in W^{1,\infty}(\Omega,\R^3)$ such that $\vc y^{k_j}$ converges to $\vc y$ weakly* in the same space. However, the energy functional is not quasiconvex and, in general, the minimum is not attained in the classical sense. Therefore, $\nabla \vc y$ is not a minimizer for $\EE(\cdot,\theta_T)$, but just of its relaxation, $\EE^{qc}(\cdot,\theta_T)$.
From a physical point of view, $\nabla \vc y$ represents the deformation gradient in the sample at a macroscopic scale, an average of the fine microstructures {%\color{red}
with gradients in $$K:=\bigcup_{i=1}^N SO(3)\mt U_i.$$}%%,$$ that can be observed in the %$x\cdot n>k$ martensite half-space.} 
It is important to remark that, in general, macroscopic deformation gradients $\nabla \vc y$ are not elements of $K$ a.e. in $\Omega$. Instead, we have $\nabla \vc y \in K^{qc}$ a.e. in $\Omega$, where
$$
K^{qc}:=
\Set{\mt M\in \R^{3\times 3}
\,\Big|\; \text{\parbox{3.in}{\centering $f(\mt M)\leq \max_K f$, for all continuous quasiconvex $f\colon\R^{3\times3}\to\R$}}}
,
$$
is the quasiconvex hull of the set $K$ (see \cite{Muller}). Characterizing the set of possible macroscopic deformations $K^{qc}$ is very important in order to fully understand the nonlinear elasticity model. 
On the other hand, the set of constant macroscopic gradients $\mt B$ which can form an interface with austenite, having constant gradient $\mt A$, is in general smaller then the whole of $K^{qc}$. Indeed, a Lipschitz function whose gradient is equal to $\mt A,\mt B$ a.e. in $\Omega$, with $\mt A,\mt B\in\R^{3\times3}$ must satisfy a generalized version of the Hadamard jump condition proved in \cite{BallJames1}: %This restricts the set of possible $A,B$, and
%which is the background of many results for compatibility between phases:
\begin{proposition}[{\cite[Prop. 1]{BallJames1}}]
\label{rank one connections}
Let $\Omega\in\R^3$ be open and connected. Assume $\vc y\in W^{1,\infty}(\Omega,\R^3)$ satisfies
\begin{align*}
\nabla \vc y(\vc x) = \mt  A,\qquad \text{a.e. }\vc x\in\Omega_A; \qquad
\nabla \vc y(\vc x) = \mt  B,\qquad \text{a.e. }\vc x\in\Omega _B,
\end{align*}
where $\mt A,\,\mt B \in \R^{3\times3}$ and $\Omega_A,\,\Omega_B$ are disjoint measurable sets such that 
$$\Omega_A\cup\Omega_B=\Omega,\qquad\meas (\Omega_A)>0,\qquad\meas(\Omega_B)>0.$$ 
Then,
$$
\mt A-\mt B= \vc a\otimes \vc n,\qquad\vc  a, \vc n\in\R^3,\,|\vc  n|=1. 
$$
\end{proposition}  
This result forces a fixed macroscopic gradient of martensite with constant gradient $\mt B$ to be rank-one connected to austenite, having constant gradient $\mt A$, across every interface between the two. Furthermore, it implies that, in this case, every phase interface is planar. For these reasons, Proposition \ref{rank one connections} is the background for many results for compatibility between phases.\\
\indent
However, this type of result fails to be true for more general gradients $\nabla \vc y\in L^\infty(\Omega;\R^{3\times3})$. As a matter of fact, as shown in \cite{BallJames1}, it is possible to construct a Lipschitz function $\vc z\in W^{1,\infty}(\Omega,\R^3)$ which is constant in the set $\Omega\cap\{\vc x\colon \vc x\cdot \vc n > c\} $ for some $c\in\R$, $\vc n\in\R^3$, and whose gradient $\nabla \vc y \in \{\mt F_1,\mt F_2,\mt F_3,\mt F_4\}$ in $\Omega\cap\{\vc x\colon \vc x\cdot \vc n< c\}$ for some matrices $\mt F_i$, such that $\nabla \vc y$ is not rank-one almost everywhere in $\Omega\cap\{\vc x\colon \vc x\cdot\vc  n< c\}$. 
Indeed a fractal behaviour of $\nabla\vc  y$ close to $\vc x\cdot \vc n =c$, finely mixing martensitic variants near the interface, allows one to achieve compatibility between incompatible gradients. That is, compatibility is achieved on the average. 
Possible approaches to recovering the Hadamard jump condition in an average sense can be found in \cite{BallCarstensen_Tobe}, or in Remark \ref{HJC_weak} below.
Another generalization of Proposition \ref{rank one connections} was proved in \cite{BallKoumatos} by assuming $\vc y\in W^{1,\infty}(\Omega,\R^3)$ to be $C^1$ both in $\overline\Omega_A$ and $\overline\Omega_B$, with $\Omega_A,\Omega_B$ two open disjoint subdomains of $\Omega$, separated by a piecewise $C^1$, possibly curved, 2-dimensional interface $\Gamma$ such that $\Omega=\Omega_A\cup\Omega_B\cup\Gamma$. % $\Gamma$.} Under these hypotheses in \cite{BallKoumatos} it is proved that a rank-one connection must hold on $\Gamma$. 
In the case of martensitic transformations, $\nabla \vc y=\mt 1$ in $\Omega_A$, while in $\Omega_B$ $\nabla \vc y$ represents a continuously varying macroscopic deformation gradient corresponding to a continuously varying martensitic microstructure. This result can be extended to $\vc y\in H^1(\Omega,\R^3)$ with $\nabla \vc y\in BV(\Omega;\R^{3\times 3})$ as done in the two dimensional setting in \cite{MullerDolzman}, or more generally in Lemma \ref{Had 1} below.
{%\color{red}
However, the deep result of \cite{Iwaniec} states that, in the case where $\vc y\in W^{1,\infty}(\R^3,\R^3)$, and $\vc y$ is constant in $\{\vc x\cdot \vc n > c\}$, the polyconvex hull of the set $\bigl\{\nabla\vc y(\vc x)\colon\vc x\in \{\vc x\cdot \vc n < c\} \bigr\}$, which contains $\bigl\{\nabla\vc y(\vc x)\colon\vc x\in \{\vc x\cdot \vc n < c\} \bigr\}^{qc}$, might not contain a matrix which is rank-one connected to $\mt 0$, the deformation gradient in $\{\vc x\cdot \vc n > c\}$.}
%
%
%
% with constant deformation gradient in $\Omega\cap \{\vc x\cdot \vc n > c\}$ given a plane $\{\vc x\cdot \vc n = c\}$ intersecting $\Omega$, the polyconvex hull of the set $\bigl\{\nabla\vc y(\vc x)\colon\vc x\in \Omega\cap \{\vc x\cdot \vc n < c\} \bigr\}$, which contains $\bigl\{\nabla\vc y(\vc x)\colon\vc x\in \Omega\cap \{\vc x\cdot \vc n < c\} \bigr\}^{qc}$, need not be rank-one connected to %the macroscopic martensitic deformation gradient might not be rank-one connected to 
%%the constant deformation gradient in austenite, even across a planar interface.
%.
\\
\indent
%{\color{red}On the other hand, as briefly explained in the next section, the authors of \cite{JamesHyst} proved that in the case of materials satisfying the cofactor conditions, it is possible to construct 
%infinitely many macrostructures which are compatible with austenite. Besides, it is possible to combine them to form macroscopic deformation gradients being just in $L^\infty(\Omega,\R^{3\times 3})$, as the ones observed in Zn\textsubscript{45}Au\textsubscript{30}Cu\textsubscript{25}.} 
%
%On the other hand, as briefly explained in the next section, the authors of \cite{JamesHyst} proved that in the case of materials satisfying the cofactor conditions, it is much easier to have macrostructures which are compatible with austenite.
%
On the other hand, in order to fully capture the complex microstructures observed in Zn\textsubscript{45}Au\textsubscript{30}Cu\textsubscript{25}, we are interested in macroscopic deformation gradients that are just in $L^\infty(\Omega,\R^{3\times 3})$. %This
Therefore, %in order not to lose generality in the study of microstructures that can possibly arise in Zn\textsubscript{45}Au\textsubscript{30}Cu\textsubscript{25} during phase transition, 
in Section \ref{Hadamard generalizations} we generalize Proposition \ref{rank one connections} to non-constant deformation gradients in $L^\infty(\Omega;\R^{3\times3})$ and to curved interfaces. However, given the above mentioned counterexamples of \cite{BallJames1,Iwaniec}, we need to change perspective and introduce some further hypotheses. This is done by recalling the idea of a moving mask explained in the introduction, which is mathematically framed in Section \ref{Sep phase}, and where the deformation gradient at a certain point $\vc x$ changes only once during phase transition, i.e., when the martensite-austenite interface passes through $\vc x$.  
\end{subsection}

%%%%%%%%%%%%%%%%%%%%%%%%%%%%%%%%%%%%%%%%%
%%%%   THE COFACTOR CONDITIONS
%%%%%%%%%%%%%%%%%%%%%%%%%%%%%%%%%%%%%%%%%
\subsection{Twinning theory and the cofactor conditions}
\label{cofac cond sec}
As explained in the previous section, the existence of a constant macroscopic martensitic deformation compatible with austenite, is related to the existence of a matrix $\mt F\in K^{qc}$ such that $\mt F=\textbf{1}+\vc a\otimes \vc n$. Conditions on the deformation parameters under which such matrices exist have been first investigated in \cite{BallJames1} in the case of two wells, i.e., $N=2$, and then generalized in \cite{BallCarstensen1,BallCarstensen2}. The case where $N=2$ is the most widely studied, as it is the only one for which an explicit characterization of $K^{qc}$ is known, and turns out to be a fundamental tool to explain a wide range of experimental observations. Therefore, we now focus on the possibility of a pair of martensitic variants forming interfaces with austenite. The notation and results of this section follow closely those in  \cite{JamesHyst}.\\

Let us first recall that given two different variants of martensite, represented by $\mt U_1,\mt U_2\in \R^{3\times3}_{Sym^+}$, there exists a rotation $\mt R\in SO(3)$ satisfying $\mt U_2=\mt {RU}_1\mt R^T$. A first useful result is the following:
\begin{proposition}[{\cite[Prop. 12]{JamesHyst}}]
Let $\mt U_1,\mt U_2\in\R^{3\times3}_{Sym^+}$ with $\mt U_1\neq\mt U_2$. % be two variants of martensite. 
Suppose further that they are compatible in the sense that there is a matrix $\hat {\mt R}\in SO(3)$ such that
\beq
\label{compatib condit}
\hat{\mt R}\mt U_2-\mt U_1= \vc b\otimes \vc m,
\eeq
$\vc b$, $\vc m\in\R^3$. Then there is a unit vector $\hat{\vc e}\in\R^3$ such that
\beq
\label{e hat}
\mt U_2 = (-\mt {1}+2\hat{\vc e}\otimes \hat{\vc e})\mt U_1(-\mt {1}+2\hat{\vc e}\otimes \hat{\vc e}).
\eeq
Conversely, if \eqref{e hat} is satisfied, then there exist $\hat{\mt R}\in SO(3)$, $\vc b,\vc m\in\R^3$ such that \eqref{compatib condit} holds.
\end{proposition}
Equation \eqref{compatib condit} is called the compatibility condition for two variants of martensite; the solutions to this equation can be classified into three categories: compound, type I and type II twins. It is possible to prove that (see e.g., \cite{Batt}), once $\mt U_1$ and $\mt U_2$ are given and \eqref{e hat} holds, the compatibility condition always has two solutions $(\hat{\mt  R}_I, \vc b_I\otimes \vc m_I)$ and $(\hat {\mt R}_{II},\vc b_{II}\otimes \vc m_{II})$. The solutions can be expressed as follows:
\begin{align}
\label{10}
&\text{type I}\qquad &&\vc m_I = \hat{\vc e}, \qquad &&\vc b_{I}=2\Big(\frac{\mt U_1^{-1}\hat{\vc e}}{|\mt U_1^{-1}\hat{\vc e}|^2}-\mt U_1\hat{\vc e}\Big),\\
\label{11}
&\text{type II}\qquad &&\vc m_{II} = 2\Big(\hat{\vc e}-\frac{\mt U_1^{2}\hat{\vc e}}{|\mt U_1\hat{\vc e}|^2}\Big),\qquad &&\vc b_{II}=\mt U_1\hat{\vc e},
\end{align}
where $\hat{\vc e}$ is as in \eqref{e hat}. 
If $\hat{\vc  e}$ satisfying \eqref{e hat} is unique up to change of sign, the two solutions \eqref{10} and \eqref{11} of \eqref{compatib condit} are called type I and type II twins respectively. In case there exist two different non-parallel unit vectors satisfying \eqref{e hat}, the resulting pair of solutions \eqref{10}--\eqref{11} are called compound twins. Nonetheless,
it is possible to prove (see e.g., \cite{Batt}) that in the case of compound twins, given two different unit vectors satisfying \eqref{e hat}, namely $\hat{\vc e}_1$ and $\hat{\vc e}_2$, then
$$
\vc b_I^1\otimes \vc m_I^1 :=2\Big(\frac{\mt U_1^{-1}\hat{\vc e}_1}{|\mt U_1^{-1}\hat{\vc e}_1|^2}-\mt U_1\hat{\vc e}_1\Big)\otimes\hat{\vc e}_1= \mt U_1\hat{\vc e}_2 \otimes 2\Big(\hat{\vc e}_2-\frac{\mt U_1^{2}\hat{\vc e}_2}{|\mt U_1\hat{\vc e}_2|^2}\Big)=:\vc b_{II}^2\otimes \vc m_{II}^2.
$$
Therefore, there are just two solutions to \eqref{e hat}, even in the case of compound twins, each of which can be considered as both a type I and a type II twin. Below, however, when we refer to type I or type II solutions of \eqref{compatib condit} we assume implicitly that they are not compound solutions. Furthermore, we sometimes abuse of notation and write that $\mt U_1,\mt U_2$ form a compound twin if the solutions of the twinning equations \eqref{compatib condit} are compound twins.
The following characterization of compound twins is used below:
\begin{proposition}[{\cite[Prop. 1]{JamesHyst}}]
\label{comp dom char}
Let $\mt U_1$ and $\mt U_2$ be two different variants of martensite and $\hat{\vc  e}_1$ a unit vector such that \eqref{e hat} is satisfied. Then there exists a second unit vector $\hat {\vc e}_2$ not parallel to $\hat {\vc e}_1$ satisfying \eqref{e hat} if and only if $\hat {\vc e}_1$ is perpendicular to an eigenvector of $\mt U_1$. When this condition is verified, $\hat {\vc e}_2$ is unique up to change of sign and is perpendicular to both $\hat{\vc  e}_1$ and that eigenvector.
\end{proposition}

Let us now consider a simple laminate, i.e., a constant macroscopic gradient $\nabla \vc y$ equal a.e. to $\lambda \hat {\mt R}\mt U_2+(1-\lambda)\mt U_1$ for some $\lambda\in(0,1)$ and some rank-one connected $\mt {RU}_2,\mt U_1\in K$. Following \cite{BallJames1,JamesHyst} we focus on the possibility for such $\nabla \vc y$
%two compatible martensitic variants $U_1$ and $U_2$, 
to be compatible with austenite. By Proposition \ref{rank one connections}, %an example can be constructed when there exists  
a necessary condition is that $SO(3)$ has 
a rank-one connection with $\lambda \hat {\mt R}\mt U_2+(1-\lambda)\mt U_1$. The existence of $(\mt R,\lambda,\vc a\otimes \vc n)$ solving
\beq
\label{auste marte}
\mt R\bigl[\lambda \hat{\mt  R}\mt U_2+(1-\lambda)\mt U_1\bigr]-\mt{1}=\mt R\bigl[\lambda(\mt U_1+\vc b\otimes \vc m)+(1-\lambda)\mt U_1\bigr]-\mt {1} = \vc a\otimes \vc n,
\eeq
%that is a twinned laminate compatible with austenite, exists. 
that is a twinned laminate compatible with austenite, was first studied in \cite{Wechsler} and later in \cite{BallJames1}. 
Lattice deformations and parameters of materials that are usually considered in the literature lead to twins with 
%Twins in most of the materials studied in the literature have 
exactly four solutions to equation \eqref{auste marte}. Nonetheless, in some cases the number of solutions can be just zero, one or two, and, under some particular condition on the lattice parameters, as in the case of the material discovered in \cite{JamesNew}, \eqref{auste marte} is satisfied for all $\lambda \in [0,1]$. The following result gives necessary and sufficient conditions for this to hold:
\begin{theorem}[{\cite[Thm. 2]{JamesHyst}}]
\label{thm cof cond}
Let $\mt U_1,\mt U_2\in \R^{3\times3}_{Sym^+}$ be distinct and such that there exist $\hat {\mt R}\in SO(3)$ and $\vc b,\vc m\in\R^3$ satisfying
%\beq
%\label{compat thm eq}
$$
\hat{\mt  R} \mt U_2 =\mt  U_1 + \vc b\otimes \vc m.
$$%\eeq
Then, \eqref{auste marte} has a solution $\mt R\in SO(3)$, $\vc a,\vc n\in \R^3$ for each $\lambda\in [0,1]$ if and only if the following \textbf{cofactor conditions} hold:
\begin{description}
\item[(CC1)] The middle eigenvalue $\lambda_2$ of $\mt U_1$ satisfies $\lambda_2=1$,
\item[(CC2)] $ \vc b\cdot \mt U_1\cof (\mt U_1^2-\mt {1})\vc m=0$,
\item[(CC3)] $\tr \mt U_1^2-\det \mt U_1^2-\frac{1}{4}|\vc b|^2|\vc m|^2-2\geq 0.
$
\end{description}
\end{theorem}

In the last part of this section, we report some results from \cite{JamesHyst} related to the cofactor conditions in type I/II twins.
%{\color{green}
%The following two theorems are very important, as they prove that it is possible to 
%construct non-constant macroscopic deformation gradients $\nabla \vc y$ which are rank-one connected with the identity %austenite on the phase interface 
%and just take two values $\mt {A,B}\in K$ (and not in $K^{qc}$) almost everywhere%}
%%
% remove the transition layer in the case of interfaces between austenite and Type I and Type II twins respectively.
\begin{theorem}[{\cite[Thm. 7]{JamesHyst}}]
\label{CC for I}
Let $\mt U_1,\mt U_2\in \R^{3\times3}_{Sym^+}$ be distinct and such that $\hat {\mt R}\in SO(3)$, $\vc b_{I},\vc m_{I}\in\R^3$ is a type I solution to \eqref{compatib condit}. Suppose further that $\mt U_1,\vc b_{I},\vc m_{I}$ satisfy the cofactor conditions. Then, there exist $\mt R_0\in SO(3)$, $\vc a_0\in\R^3$, $\vc n_0,\vc n_1 \in \mathbb{S}^2$ and $\xi\neq 0$ such that
%
%are particular choices of $\omega$, $\omega^* \in \{\pm\}$ such that $\mt R_1^{\omega^*} = \mt R^\omega_0$ and $\vc a_1^{\omega^*}=\xi \vc a_0^\omega$ for some $\xi\neq 0$, so that
\beq
\label{thm cc1 a}
\mt R_0 \mt U_1= \mt {1} + \vc a_0\otimes \vc n_0,\qquad  \mt R_0 ( \mt U_1 + \vc b_I\otimes \vc m_I)= \mt {1} + \vc a_0\otimes \xi \vc n_1.
\eeq
Furthermore,% one of the two families of solutions of the crystallographic theory can be written
\beq
\label{forall lambda}
\mt R_0 [\mt  U_1 +\lambda \vc b_I\otimes\vc  m_I]= \mt {1} + \vc a_0\otimes \bigr(\lambda \xi \vc n_1+(1-\lambda)\vc n_0\bigl),\qquad  \text{for all } \lambda\in[0,1].
\eeq
%The three deformation gradients $\mt {1}$, $\mt R_0 \mt U_1$, $\mt R_0 \hat{\mt  R} \mt U_2$ can form a compatible austenite-martensite triple junction in the sense that
%$$
%\mt R_0 \mt U_1-\mt {1}=\vc a_0\otimes \vc n_0,\qquad \mt R_0 \hat{\mt  R} \mt U_2-\mt {1}=\vc a_0\otimes \xi \vc n_1,\qquad \mt R_0\hat {\mt R} \mt U_2-\mt R_0 \mt U_1=\mt R_0 \vc b_I\otimes\vc  m_I.
%$$
%There is a constant $c\neq 0$ such that $c\vc m_I=\xi \vc n_1-\vc n_0$, so that the three vectors $\vc n_0$, $\vc n_1$ and $\vc m_I$ lie in a plane.
\end{theorem}

\begin{theorem}[{\cite[Thm. 8]{JamesHyst}}]
\label{CC for II}
Let $\mt U_1,\mt U_2\in \R^{3\times3}_{Sym^+}$ be distinct and such that $\hat{\mt  R}\in SO(3)$, $\vc b_{II},\vc m_{II}\in\R^3$ is a type II solution to \eqref{compatib condit}. Suppose further that $\mt U_1,\vc b_{II},\vc m_{II}$ satisfy the cofactor conditions. Then, there exist $\mt R_0\in SO(3)$, $\vc a_0,\vc a_1\in\R^3$, $\vc n_0\in \mathbb{S}^2$ and $\xi\neq 0$ such that
%
%Then, there are particular choices of $\omega$, $\omega^* \in \{\pm\}$ such that $\mt R_1^{\omega^*} = R_0$ and $\vc n_1^{\omega^*}=\xi \vc n_0^\omega$ for some $\xi\neq 0$, so that
\beq
\label{thm cc1 b}
\mt R_0 \mt U_1= \mt {1} + \vc a_0\otimes\vc n_0,\qquad \mt R_0 ( \mt U_1 + \vc b_{II}\otimes\vc  m_{II})= \mt {1} + \xi \vc a_1 \otimes \vc n_0.
\eeq
Furthermore, %one of the two families of solutions of the crystallographic theory can be written
\beq
\label{forall lambda 2}
\mt R_0 [ \mt U_1 +\lambda \vc b_{II}\otimes \vc m_{II}]= \mt{1} + \bigr(\lambda \xi \vc a_1+(1-\lambda)\vc  a_0\bigl)\otimes \vc n_0
,\qquad \text{for all } \lambda \in[0,1].
\eeq
%The normal $\vc n_0$ to the austenite-martensite interface is independent of the volume fraction $\lambda$ and is parallel to the martensite-martensite interface normal $\vc m_{II}$.%^\omega = cm_{II}$ for some $c\neq 0$.
\end{theorem}

%%%%%%%%%%%%%%%%%%%%%%%%%%%%%%%%%%%%%%%%%%%%%%%%%%%
%%%%% Rectifiable sets
\subsection{Some preliminaries on $k$-rectifiable sets}
\label{Some preliminaries on $k$-rectifiable sets}
%%%%%%%%%%%%%%%%%%%%%%%%%%%%%%%%%%%%%%%%%%%%%%%%%%%
In this section we recall some standard results on Lipschitz functions and $k$-rectifiable sets from \cite{AFP,Federer,Morgan} (see also \cite{ABC} for properties of level sets of Lipschitz functions). We denote by $\mathscr H^k$ the $k$-dimensional Hausdorff measure, and write $\mathscr H^k\mres E$ for its restriction to an $\mathscr H^k$ measurable subset $E$. $C_c(\R^d)$ stands for the space of continuous functions with compact support in $\R^d$, while $B^d(\vc x,r)$ denotes the $d$-dimensional ball centred at $x$, with radius $r$ and of volume $\omega_dr^d$. %In what follows, we call a Lipschitz $k$-graph the graphs of a Lipschitz functions of $k$ variables. That is given an open and connected set $\omega\subset\R^k$, a Lipschitz function $\psi\colon \omega\to\R$, and a rotation $\mt Q\in SO(3)$, we call $k$-graph the set of points $Q$ $(\vc x,\psi(\vc x))$ 
We start with the following definitions:
{%\color{red}
\begin{definition}
A Lipschitz $k$-graph $\mathcal G$ is a set of points in $\R^d$ with $d>k$ such that there exists an open and connected set $\omega\subset\R^k$, a Lipschitz map $\boldsymbol\psi\colon \omega\to\R^{d-k}$, and a rotation $\mt Q\in SO(d)$ satisfying
$$
\mathcal G:=\bigl\{%\vc x\in\R^d\colon \vc x = 
\mt Q\vc x,\, \vc x = (\vc x',\boldsymbol\psi(\vc x')),\,\vc x'\in\omega	\bigr\}.
$$
\end{definition}
}
\begin{definition}
Let $E\subset\R^d$ be an $\mathscr H^k$-measurable set satisfying $\mathscr H^k(E)<\infty$. We say that $E$ is $k$-rectifiable if there exist countably many Lipschitz mappings $\vc f_i\colon\R^k\to\R^d$ such that
$$
\mathscr H^k\Bigl( E\setminus\bigcup_{i=1}^\infty \vc f_i(\R^k)\Bigr)=0.
$$
\end{definition}
An equivalent characterization for such sets is given by the following result:
\begin{proposition}[{\cite[Prop. 2.76]{AFP}}]
\label{SIPUO}
Any $\mathscr H^k$-measurable set $E$ is countably $\mathscr H^k$-rectifiable if and only if there exist countably many Lipschitz $k$-graphs $\mathcal G_i\subset\R^N$, such that 
$$
\mathscr H^k\Bigl(E\setminus\bigcup^\infty_{i=1}\mathcal G_i\Bigr)=0.
$$
\end{proposition}

In what follows, a particular case of \cite[Theorem 3.2.22]{Federer} is also used:
\begin{theorem}
\label{Fed thm}
Let $\Omega\subset \R^3$ be open, bounded and connected, $\mathcal Z\subset \R^3$ be a $1$-rectifiable set and $\vc f\colon \Omega\to \mathcal Z$ a Lipschitz function. Define the $1$-dimensional Jacobian of $\vc f$ by:
$$
J_1 \vc f:=\sqrt{\sum_{i,j}(\nabla\vc f)_{ij}^2}.
$$
Then:
\begin{itemize}
\item for $\mathscr{L}^3$ almost every $\vc x\in\Omega$, either $J_1\vc f(\vc x) = 0$, or the image of $\nabla\vc f(x)$ is a $1$-dimensional vector space, i.e., $\rank \nabla\vc f(\vc x)\leq 1$, $\mathscr L^3$-almost everywhere in $\Omega$;
\item for $\mathscr H^1$ almost all $\boldsymbol\xi\in\mathcal Z$, $\vc f^{-1}(\boldsymbol\xi)$ is $2$-rectifiable;% and $\mathscr{H}^2$-measurable;
\item for every integrable $g\colon\Omega\to [-\infty,\infty]$
\beq
\label{coarea}
\int_\Omega g(\vc x)J_1\vc f(\vc x)\,\mathrm{d}\vc x
=\int_{\mathcal Z}\int_{\vc f^{-1}(\boldsymbol\xi)\cap\Omega}g(\vc s)\,\mathrm{d}\mathscr H^2(\vc s)\, \mathrm{d}\mathscr H^1(\boldsymbol\xi)
\eeq
\end{itemize} 
\end{theorem}

%%%%%%%%%%%%%%%%%%%%%%%%%%%%%%%%%%%%%%%%%%%%%%%%%%%%%%%%%%%
%%%%  Sezione sul modello

\section{The moving mask assumption}
\label{Sep phase}
{The aim of this section is to give a precise definition of the moving mask assumption (see Definition \ref{defi mi} below), and to frame it in the context of dynamics for nonlinear continuum mechanics. This is done by recalling first the simplified model derived in \cite{FDP1} to describe the evolution of martensitic transformation in the context of nonlinear continuum mechanics. In this framework, we introduce some hypotheses approximating experimental observation. These hypotheses are made precise in Definition \ref{defi mi}. We remark that the model in \cite{FDP1} is used here just to frame the moving mask assumption, and that the rest of the paper relies on Definition \ref{defi mi} only, which could be hence taken by the reader as a standalone assumption. \\
}

In \cite{FDP1} we introduced a continuum model for the evolution of martensitic transformations. After passing to the limit in which the elastic constants tend to infinity and the interface energy density tends to zero, we deduced that the deformation gradients and the temperature field generate in the limit a Young measure $\nu_{\vc x,t}$ (see as a reference \cite{Muller,Pedregal}) and a function $\theta$ satisfying in a suitable sense
\begin{align}
\label{equazPbEv1}
\rho_0\theta_t - d\Delta \theta &=  - \theta_T\frac{\partial}{\partial t}\int_{\R^{3\times3}}\eta_1(\mt A)\,\mathrm{d}\nu_{\vc x,t}(\mt A),\qquad&\text{a.e. in $\Omega\times(0,T)$},\\
\label{equazPbEv2}
\supp \nu_{\vc x, t}&\subset SO(3)\cup K,\qquad  &\text{a.e. in $\Omega\times(0,T)$},
\end{align}
complemented with some initial and boundary conditions. 
Here, $\rho_0$ is the density of the body, $d$ is a diffusivity coefficient which is supposed to be constant, and $\eta_1$ is a smooth function such that 
$$
\eta_1(\mt F) = 0, \quad\text{for all $\mt F\in SO(3)$},\qquad \eta_1(\mt F) = -\frac\alpha{\theta_T}, \quad\text{for all $\mt F\in K$},
$$
for some constant $\alpha>0$ representing the latent heat of the transformation. This system of equations is underdetermined, and should therefore be closed with some constitutive relation between $\int_{\R^{3\times3}}\eta_1(\mt A)\,\mathrm{d}\nu_{\vc x,t}(\mt A)$, and $\theta$ and $\nu_{\vc x,t}$. Nonetheless, we aim to characterise solutions independently of the constitutive relation. In order to do this we introduce some hypotheses on the solutions, that are based on experimental observation and that together we call the \textit{moving mask} approximation, defined precisely in Definition \ref{defi mi} below, using the following ingredients: 
\begin{itemize}
\item the phases are separated, that is there exist open sets $\Omega_A(t),\Omega_M(t)\subset\Omega$ such that
$$
\Omega_A(t)\cap\Omega_M(t)=\varnothing,
\qquad \mathscr{L}^3 \bigl( \Omega\setminus(\Omega_A(t)\cup\Omega_M(t))\bigr)=0, \quad\text{a.e. $t\in(0,T)$},
$$
and
%{\color{red}
\begin{align*}
\nu_{\vc x,t}(SO(3))= 1,\qquad\text{a.e. $\vc x\in\Omega_A(t)$, a.e. $t\in(0,T)$}, \\
\nu_{\vc x,t}(K)= 1,\qquad\text{a.e. $\vc x\in\Omega_M(t)$, a.e. $t\in(0,T)$}.
\end{align*}
The domain can hence be divided for almost every $t\in(0,T)$ into two regions, the region with martensite $\Omega_M(t)$, and the region with austenite $\Omega_A(t)$. Thus, 
\beq
\label{separato}
\int_{\R^{3\times3}}\eta_1(\mt A)\,\mathrm{d}\nu_{\vc x,t}(\mt A) = -\frac{\alpha}{\theta_T} \chi_{\Omega_M}(\vc x,t),
\eeq
where $\chi_{\Omega_M}(\vc x,t)$ is the characteristic function of $\Omega_M(t)$. The austenite-martensite phase boundary is sharp in this case. %} 
In terms of macroscopic deformation gradients this reads
\[%beq
%\label{cons mom final}
\begin{split} &\nabla \vc y (\vc x,t) \in K^{qc},\quad\text{ a.e. in $\Omega_M(t)$, a.e. $t\in(0,T)$},\\ 
&\nabla \vc y (\vc x,t) \in SO(3),\quad\text{ a.e. in $\Omega_A(t)$, a.e. $t\in(0,T)$},\\
%\Omega_A(t),\Omega_M(t) \text{ open,}&
%\qquad\Omega_A(t)\cap\Omega_M(t)=\varnothing,
%\qquad \mathscr{L}^3 \bigl( \Omega\setminus(\Omega_A(t)\cup\Omega_M(t))\bigr)=0, \quad\text{a.e. $t\in(0,T)$};
%\end{align*}
\end{split}
\]%eeq
\item during the phase transition, the macroscopic deformation gradient remains equal to a constant rotation in the austenite region. This is the case, for example, when the austenite region is connected;
\item the phase interface moves continuously. More precisely, for almost every point $\vc x$ in the domain, there exists a time when $\vc x$ is contained in the phase interface (see also \cite[Remark 5.2]{FDP1});
\item the microstructures do not change after the transformation has happened. This assumption makes particular sense in the context of materials satisfying the cofactor conditions, where austenite and finely twinned martensite can be exactly compatible across interfaces, and even more in Zn\textsubscript{45}Au\textsubscript{30}Cu\textsubscript{25} where the phase transition has very low thermal hysteresis and thermal expansion is hence negligible. 
\end{itemize}
% In this way, \eqref{cons ener limit} can be rewritten as
%\beq
%\label{cons ener final}
%\Tint\mint \tilde\xi\mt C\nabla\theta_k\cdot\nabla\tilde\psi \,\mathrm{d}\vc x = \Tint\mint \biggl(-\alpha\chi_{\Omega_M}+\rho_0\theta_k\biggr)\dot{\tilde\xi}\tilde\psi \,\mathrm d\vc x\mathrm dt+	\tilde\xi(0)\mint\tilde\psi\theta_0\,\mathrm d\vc x,
%\eeq
%for each $\tilde\xi\in C^1_c([0,T))$, $\tilde\psi\in C^1(\Omega)\cap H^1_D $, and
%to be coupled with{%\color{red}
%%\begin{align*}
%}%Note that $\dot{\chi}_{\Omega_M} \in L^2(0,T;H^{-1}_D)$ because of Proposition \ref{prop limit model}. 
As remarked in the introduction, this construction reflects the idea of a moving mask that uncovers a martensitic microstructure, as can be seen in the video of \cite{JamesNew}. Mathematically we can define the moving mask approximation as follows:
\begin{definition}
\label{defi mi}
{%\color{red} 
We say that $\nabla\vc y\in L^\infty(\Omega;\R^{3\times3})$ satisfies the moving mask approximation if 
\begin{itemize}
\item for each $t\in[0,T]$ there exist $\Omega_M(t),\Omega_A(t)\subset\Omega$ disjoint and open, such that $$\mathscr{L}^3 \bigl( \Omega\setminus(\Omega_A(t)\cup\Omega_M(t))\bigr)=0;$$
\item either
\[
\Omega_A(t_2)\subset\Omega_A(t_1),\qquad\text{ for all $0\leq t_1\leq t_2\leq T$},
\]
or
\[
\Omega_A(t_1)\subset\Omega_A(t_2),\qquad\text{ for all $0\leq t_1\leq t_2\leq T$};
\]
\item for a.e. $\vc x\in\Omega$ there exists $t=t(\vc x)\in[0,T]$ such that $\vc x \in \overline{\Omega}_A(t)\cap\overline{\Omega}_M(t)$;
\item there exists $\mt Q\in  SO(3)$ such that for every $t\in[0,T]$ the map $\vc y_M(\cdot,t)$ satisfying
$$
\nabla \vc y_M(\vc x,t)= 
\begin{cases}
\nabla \vc y (\vc x), \quad &\text{a.e. in }\Omega_M(t)\\
\mt Q, \quad &\text{a.e. in }\Omega_A(t)\\
\end{cases}
$$
is in $W^{1,\infty}(\Omega;\R^3).$
\end{itemize} 
}
\end{definition}
%This reflects the idea that the martensite microstructures do not change after the phase transition has happened, and do not change due to thermal expansion or mechanical stresses.\\
\begin{remark}
\rm
We note that, in the case $\Omega_M(s)\subset\Omega_M(t)$ for each $s,t\in[0,T]$ with $s<t$, we have
$$
\bigcap_{t\in[0,T]}\Omega_A(t) = \varnothing,\qquad \bigcup_{t\in [0,T]}\Omega_M(t) =\Omega.
$$
\end{remark}
\begin{remark}
\rm
If we assume \eqref{separato}, then the formula for differentiation of integrals on time dependent domains implies that
\beq
\label{derivata materiale}
\begin{split}
\Bigl\langle \frac\partial{\partial t}\int_{\R^{3\times3}}\eta_1(\mt A)\,\mathrm{d}\nu_{\vc x,t}(\mt A),\psi\Bigr\rangle =\langle \dot\chi_{\Omega_M}(\nabla \vc y),\psi\rangle = \dert \int_{\Omega_M} \psi \,\mathrm d\vc x\\ = \int_{\Gamma(t)} (\vc v \cdot \vc n) \psi\,\mathrm d\mathscr H^2,\qquad\forall \psi \in C^\infty_0(\Omega),
\end{split}
\eeq
provided $\Omega_A(t),\Omega_M(t)$ %is $2$-rectifiable 
and $\vc v\cdot \vc n$ are smooth enough (see e.g., \cite{Flanders}). %is in $L^1(\Gamma(t))$ for a.e. $t$.
Here $\Gamma(t):=\Omega\setminus(\Omega_A\cup\Omega_M)(t)$ is a surface separating $\Omega_A(t)$ from $\Omega_M(t)$, $\vc n$ denotes the outer normal to $\Omega_M$ and $\vc v(\vc s)$ is the velocity of the interface at the point $\vc s\in \Gamma(t)$ at time $t$. By $\langle \cdot,\cdot\rangle$ we denoted the duality pairing between a distribution and a test function. {A version of \eqref{derivata materiale} in the case of some solutions to \eqref{equazPbEv1}--\eqref{equazPbEv2} satisfying the moving mask assumptions is given by Corollary \ref{derivata materiale coroll} below.}
%a case of interest can be found in Corollary \ref{derivata materiale coroll} below. It thus appears clear from \eqref{derivata materiale} that the model needs to be closed with some constitutive relation for $\vc v\cdot \vc n$ in terms of $\nabla\vc y,\theta$, and possibly other quantities as the curvature of the phase interface.
\end{remark}
\section{Generalized Hadamard conditions}
\label{Hadamard generalizations}
%%%%%%%%%%%%%%%%%%%%%%%%%%%%%%%%%%%%%%%%%%%%%%%%%%%%%%
%Clearly, studying directly problem \eqref{cons ener final}--\eqref{cons mom final} is very difficult. For this reason 
In this section we restrict our attention to deformation gradients $\nabla \vc y$ that satisfy the moving mask approximation as stated in Definition \ref{defi mi}.
%
%remain constant after the phase transition, and do not change due to thermal expansion or mechanical stresses. This assumption makes particular sense in the context of materials satisfying the cofactor conditions, where austenite and finely twinned martensite can be exactly compatible across interfaces, and even more in Zn\textsubscript{45}Au\textsubscript{30}Cu\textsubscript{25} where the phase transition has very low thermal hysteresis and thermal expansion is hence negligible. As remarked in the introduction, this construction reflects the idea of a moving mask that uncovers a martensitic microstructure, as can be seen in the video of \cite{JamesNew}. 
In order to say something more about solutions under these assumptions, we prove below a variant of the Hadamard jump condition reflecting this hypothesis. In what follows, we restrict, without loss of generality, to the case $\Omega_M(s)\subset\Omega_M(t)$ for every $s<t$. As before, below $\Omega\subset\R^3$ is an open bounded connected set with Lipschitz boundary. {For simplicity, rather than working with the deformation map $\vc y$, in this section we mostly work with the displacement map $\vc z := \vc y - \mt Q\vc x$, where $\mt Q$ is as in Definition \ref{defi mi}.\\ 

We start by proving the result when the phase interfaces are planar. This situation describes, for example, the propagation of a simple martensitic laminate in the austenite phase.
}
\begin{proposition}
\label{plane Hadamard}
Let $\Gamma(t)$ be a family of parallel planes perpendicular to $\vc n\in \mathbb{S}^2$, $$\Gamma(t):=\bigl\{\vc x\in\R^3\colon \vc x\cdot  \vc n = h(t)\bigr\},$$
for some non-decreasing function $h\in C([0,T])$ satisfying $$h(0)=\inf_{\vc x\in \Omega} \vc x\cdot\vc n,\qquad h(T)=\sup_{\vc x\in \Omega} \vc x\cdot\vc n.$$ 
For $t \in [0,T]$ define%, let $\Omega_M(t),\Omega_A(t)\neq\varnothing$ be such that
$$
\Omega_M(t):=\Omega\cap\bigl\{\vc x\cdot\vc  n < h(t)\bigr\},\qquad\Omega_A(t):=\Omega\cap\bigl\{\vc x\cdot \vc n > h(t)\bigr\}.
$$
%and $\Omega_M(0)=\Omega_A(T)=\varnothing$. 
Let $\vc z\in W^{1,\infty}(\Omega;\mathbb{R}^3)$ be such that $\vc Z=\vc Z(\vc x,t)$ satisfying
\begin{equation}
\label{Z cond}
\nabla \vc Z(\vc x,t)= 
\begin{cases}
\nabla \vc z(\vc x), \quad &\text{a.e. in }\Omega_M(t)\\
\mt 0, \quad &\text{a.e. in }\Omega_A(t)\\
\end{cases}
\end{equation}
is in $W^{1,\infty}(\Omega;\mathbb{R}^3)$ for a.e. $t\in (0,T)$. Then, 
\begin{enumerate}
\item\label{numero 1 lista} there exists $ \vc a\in L^{\infty}(\Omega;\mathbb{R}^3)$ such that 
\[
\nabla  \vc z(\vc x)= \vc a(\vc x)\otimes \vc n,\qquad \text{a.e. }\vc x \in \Omega.
\]
\item if $\Omega\cap\Gamma(t)$ is connected for every $t\in(0,T)$ then $\vc z=\vc f(\vc x\cdot \vc n)$ for some $\vc f\in W^{1,\infty}((0,T);\R^3)$.
%
%there exists $\vc f\in W^{1;\infty}((h(0),h(T)),\mathbb{R}^3)$ such that $\vc z=\vc f(\vc x\cdot \vc n)$ and $\vc a=\vc f'$ a.e. in $(h(0),h(T))$.
\end{enumerate}
%Furthermore, . 
\end{proposition}
\begin{proof}
By rotating the system of coordinates we can assume without loss of generality that $\vc n=\vc e_3$. 
Let us consider the set ${\mathcal B}_1\subset\Omega$ of points where $\vc z$ is differentiable, and the set ${\mathcal B}_2$ of points $\vc x\in\Omega$ such that there exists $t^*\in(0,T)$ for which $\vc x\in \Gamma(t^*)$ and $\vc Z(\cdot,t^*) \in W^{1,\infty}(\Omega;\R^3)$. By continuity of $h$ we have that $\meas\bigl(\Omega\setminus({\mathcal B}_1\cap{\mathcal B}_2)\bigr)=0$. Let us thus consider a generic point $\hat{\vc x} \in {\mathcal B}_1\cap {\mathcal B}_2$, and notice that, since $\Omega$ is open, there exists $r>0$ such that the ball $B(\hat{\vc x},r)\subset\Omega$. By \eqref{Z cond}, $\vc Z(\cdot,t^*)$ must be constant in each connected component of $\Omega_A(t)$. In particular, as $\vc Z(\cdot,t^*)\in W^{1,\infty}(\Omega,\mathbb{R}^3)$ is  continuous, it must be a constant on $\Gamma(t^*)\cap B(\hat{\vc x},r)$. At the same time, continuity of $\vc z$ and $\vc Z(\cdot,t^*)$ implies also $\vc z(\vc x) = \vc Z(\vc x,t^*)$ for every $\vc x\in \Gamma(t^*)\cap B(\hat{\vc x},r)$. Therefore, 
%as $\vc Z(\cdot,t^*)\in W^{1,\infty}(\Omega,\mathbb{R}^3)$ is  continuous on $\Gamma(t^*)$, 
the function $\vc z(\vc x)$ must be constant on $\Gamma(t^*)\cap B(\hat{\vc x},r)$. 
%
%a generic point $\vc x\in\Omega$ and notice that as $h$ is continuous, there exists a time $t^*=t^*(\vc x)\in (0,T)$ such that $\vc x\in\Gamma(t^*)$. Furthermore, as $\vc Z(\cdot,t^*)\in W^{1,\infty}(\Omega,\mathbb{R}^3)$ is continuous on $\Gamma(t^*)$ the function $\vc z(\vc x)$ must be constant on each connected component of $\Gamma(t^*)\cap\Omega$. On the other hand, as $\Omega$ is open, there exists an open ball $B=B(\vc x,r)$ contained in $\Omega$ and $z(\vc x+s\vc e_i)=\vc z(\vc x)$ for $i=1,2$ and $|s|<r$. 
This implies, 
$$
\frac{\partial \vc z}{\partial x_i}(\hat{\vc x})=0,\qquad i=1,2.
$$ 
The arbitrariness of $\hat{\vc x}\in\mathcal{B}_1\cap\mathcal{B}_2$ yields the first statement. %\eqref{numero 1 lista}.
%As this holds for every $\vc x\in\Omega$ and $\vc z$ is Lipschitz, and hence differentiable a.e., we deduce 
%1. 
On the other hand, if $\Omega\cap\Gamma(t)$ is connected for a.e. $t\in(0,T)$, then $\vc z(\vc x)$ is constant on $\Gamma(t^*)\cap\Omega$ for a.e. $t\in(0,T)$ and hence $\vc z=\vc z(x_3)$. This concludes the proof.%This implies the existence of $\vc f\in W^{1,\infty}((0,T),\mathbb{R}^3)$ such that $\vc z=\vc f(\vc x\cdot \vc e_3)$, and, therefore, for every Lebesgue point of $\vc f$ in $(0,T)$, there exists $\vc f'$ and 
%$$
%\nabla \vc z (\vc x)= \vc f'(\vc x\cdot  \vc e_3)\otimes \vc e_3.
%$$
\end{proof}

\begin{remark}
\label{sono levelset i piani}
\normalfont
We could replace the hypothesis concerning the connectedness of $\Omega_A(t)$ by assuming that $\vc Z(\vc x,t)$ is equal to a constant $\vc c(t)\in\R^3$ in $\Omega_A(t)$ for a.e. $t\in [0,T]$. Both these assumptions are automatically satisfied if $\Omega$ is convex.
\end{remark}
\begin{remark}
\normalfont
\label{level set rk}
{In the case $\vc z=\vc f(\vc x\cdot \vc n)$, and $\Gamma(t)$ is a single plane, the phase interfaces must coincide with the level sets of $\vc f$. Therefore, given an experimentally measured martensitic macroscopic deformation gradient, under the assumption that it satisfies the moving mask approximation, and is of the form $\mt 1+\vc a(\vc x\cdot\vc n)\otimes \vc n$, for some $\vc a\in\R^3,\vc n\in \mathbb{S}^2$, we can reconstruct the position of  austenite-martensite phase interfaces, by taking the level sets of $\vc f (\vc x\cdot\vc n)= \int_0^{\vc x\cdot n}\vc a(s)\,\mathrm{d}s.$
%
%we can look at $\Gamma(t)$ as the level sets of $\vc f$. Indeed, for every $t\in [0,T]$ where \eqref{Z cond} is satisfied, $\vc z$ must be constant on $\Gamma(t)$. 
Furthermore, in the case $\vc z=\vc f(\vc x\cdot \vc n)$, and $\Gamma(t)$ is a single plane, the discontinuities in the macroscopic deformation gradient can occur only across the planes $\vc x\cdot\vc n = constant$. This is, for example, the case for type II twins satisfying the cofactor conditions, for which we refer the reader to Proposition \ref{type ii inter}.}
%This implies that $\nabla \vc z$ can be discontinuous only across planes perpendicular to $\vc n$. Therefore, if we consider $\nabla \vc z$ to be piecewise constant, representing either martensitic variants or macroscopic deformations, then all the possible gradients must be compatible both between each other and with austenite across the planes normal to $\vc n$.   
\end{remark}
%\begin{remark}
%\normalfont
%The above proposition does not contradict the counterexamples of \cite{BallJames1} cited in Section \ref{nonlonear elastic}, which states that there might exist planes where $\nabla \vc z$ is not rank-one. Indeed the strong compatibility requirement of Proposition \ref{plane Hadamard} implies that the union of these planes has $\mathscr{L}^3$-measure zero.%, and hence that $\nabla z$ cannot have a fractal structure everywhere.
%\end{remark}

Proposition \ref{plane Hadamard} can be partially generalized to the case where $\Gamma(t)$ is a family of curved interfaces. As a first step, we need to introduce the concept of moving interfaces for our problem, generalizing the previous requirements on planar such interfaces.
\begin{definition}
\label{DEf inter}
%Let $\Omega\subset \mathbb{R}^3$ be open, bounded and connected. 
We say that $\Gamma(t)\subset\Omega$ is a \textit{family of moving interfaces in $\Omega$} if:
\begin{enumerate}[(i)]
\item \label{def i}there exist two families of open disjoint sets $\Omega_M(t),\Omega_A(t)\subset\Omega$ and a bounded open interval $I_T:=[0,T]$ such that for every $t$ in $I_T$, 
$$\Omega = \Omega_M(t)\cup\Omega_A(t)\cup \Gamma(t)
%,\qquad  \Omega_M(t)\cap\Omega_A(t)=\varnothing
\qquad \text{and} \qquad \Gamma(t)\cap\Omega_M(t)=\Gamma(t)\cap\Omega_A(t)=\varnothing.
%
%\qquad\Omega_M(t),\Omega_A(t)\neq\varnothing.
$$
Furthermore, $\Omega_M(t)$ is non-decreasing in t, i.e., 
$$\Omega_M(t)\subset\Omega_M(s),\qquad \Omega_A(s)\subset\Omega_A(t),\qquad \forall t<s\in I_T;$$
%\item %the set
%%$$\mathcal \mathcal C:=\bigl\{\vc x\in\Omega\colon x\in\Gamma(t) \text{ for some $t\in I_T$ with 
%$\Gamma(t)$ is a $2$-rectifiable set for almost every $t\in I_T$;
%}
%\exists t\in I_T,\text{ such that $\vc x\in\Gamma(t)$},\,\Gamma(t) \text{ is $2$-rectifiable}\bigr\}
%,$$
%is measurable and $\mathscr L^3\bigl(\Omega\setminus\mathcal C\bigr)=0$;
\item
\label{def mi 2} 
the set 
\[
\mathcal B:=
\Set{\vc x\in%\mathcal 
\Omega\,\Bigg|\; \text{\parbox{3.2in}{\centering $\vc x\in\Gamma(t^*)$, $t^*\in I_T$ and there exist $\mathcal U_\vc x\subset\Omega$ open and connected, $\vc x\in\mathcal{U}_\vc x$, and a Lipschitz $2$-graph $\mathcal{G}_\vc x$ differentiable at $\vc x$ such that $\mathcal{G}_\vc x\cap \mathcal U_\vc x \subset \Gamma(t^*)\cap\mathcal U_\vc x\subset\overline\Omega_{M}(t^*)\cap
\overline\Omega_{A}(t^*)\cap\mathcal{U}_\vc x$}}}
\]
\noindent 
is measurable and $\mathscr L^3\bigl(\Omega\setminus\mathcal B\bigr)=0$.
\end{enumerate}
Points in $\mathcal B$ are called regular points for $\Gamma(t)$. 
\end{definition}
At this point we can also introduce the concept of a regular moving mask approximation:
{
\begin{definition}
We say that $\vc y\in W^{1,\infty}(\Omega;\R^3)$ satisfies a regular moving mask approximation if it satisfies the moving mask approximation and $$\Gamma(t)=\Omega\setminus(\Omega_A(t)\cup\Omega_M(t))$$ is a family of moving interfaces in $\Omega$, where $\Omega_A,\Omega_M$ are as in Definition \rm\ref{defi mi}.
\end{definition}
}
\begin{figure}
\centering
\begin{tikzpicture}
\draw (-5,1.5) -- (5,1.5);
\draw (-5,-1) -- (5,-1);
\draw (5,-1) -- (5,1.5);
\draw (-5,-1) -- (-5,1.5);

\draw (-1,1) -- (1,1);
\filldraw [red] (0,1) circle (2pt) node[anchor=south,black] {$\vc x_a$};
\draw (-5,0) .. controls (-2,1) and (-2,-1) .. (0,0.5);
\draw (0,0.5) .. controls (1,-1) and (2,0) .. (5,0);
\filldraw [red] (0,0.5) circle (2pt) node[anchor=west,black] {$\vc x_b$};
\filldraw [red] (4.65,0) circle (0pt) node[anchor=north,black] {$\Gamma(t)$};
\filldraw [red] (1,1) circle (0pt) node[anchor=west,black] {$\Gamma(t)$};
\filldraw [red] (-3.5,-0.5) circle (0pt) node[anchor=west,black] {$\Omega_M(t)$};
\filldraw [red] (-3.5,0.7) circle (0pt) node[anchor=west,black] {$\Omega_A(t)$};
 
\end{tikzpicture}
\caption{%\ref{punti irr} 
%\ref{punti irr}
\label{punti irr}
Points which are not regular: $\vc x_a$ is in a smooth $k$-graph contained in $\Gamma(t)$, but is not separating $\Omega_A$ from $\Omega_M$. $\Gamma(t)$ does not coincide with a Lipschitz function differentiable at $\vc x_b$. %$\vc x_1$ is a regular point.
%The point $x_d$ might not satisfy, depending on the choice of the $S_i$, Definition \ref{DEf inter} (iii) d. 
}
\end{figure}
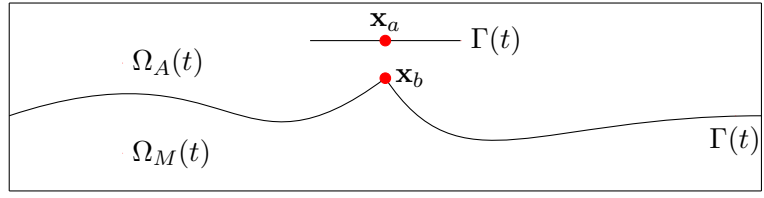
%\begin{figure}[!htbp]
%      \centering%  
%      %   \subfigure
%%	{\label{punti irr}\includegraphics[scale=0.8]{separation_points.png}}
%%	\subfigure
%		{\label{sep dots}\includegraphics[scale=1.0]{separation_points2.png}}
%\caption{%\ref{punti irr} 
%%\ref{punti irr}
%\label{punti irr}
%Points which are not regular: %$x_1$ and $x_3$ are points such that 
%%
%% is a point where a tangent plane to $\Gamma$ does not exists, 
%$\vc x_2$ is on a regular component of $\Gamma$, but is not separating $\Omega_A$ from $\Omega_M$. $\Gamma$ does not coincide with a differentiable function on a neighbourhood of $\vc x_3$. $\vc x_1$ is a regular point.
%%The point $x_d$ might not satisfy, depending on the choice of the $S_i$, Definition \ref{DEf inter} (iii) d. 
%}
%\end{figure}

\begin{remark}
\normalfont
The requirement $\Gamma(t^*)\cap\mathcal U\subset\overline\Omega_{A}(t^*)\cap\overline\Omega_{M}(t^*)\cap\mathcal{U}$ in Definition \ref{DEf inter}(\ref{def mi 2}) is mainly to guarantee that the set where an interface is cutting either $\Omega_A$ or $\Omega_M$ and not separating one from the other is small (see e.g., the point $\vc x_a$ in Figure \ref{punti irr}). In this way, families of moving interfaces satisfying the separation condition may also describe further nucleations in the interior of $\Omega_A$ during the phase transition. 
\end{remark}

%In the following useful lemma we denote by $\hd(A,B)$ the Hausdorff distance between two open sets $A,B\subset\Omega$. As $A,B$ are not closed, $\hd$ is not properly a distance, but...
%\begin{lemma}
%The families of open sets $\Omega_A,\Omega_M(t)$ satisfying Definition \ref{DEf inter} (i) satisfy
%$$
%\forall \eps>0,\,\exists\delta>0\colon \forall t,s\in I_T,\; |t-s|<\delta \Rightarrow \hd(\Omega_A(t),\Omega_A(s))<\eps,
%$$
%if and only if
%\end{lemma}
%
{Below, we say that a curve $\vc c\colon [t_0,t_1]\to\R^3$, for some $t_0,t_1\in\R$, is simple if $\vc c(s)\neq \vc c(t)$ for each $s,t\in[t_0,t_1]$.} The following theorem generalizes Proposition \ref{plane Hadamard} to curved interfaces:
\begin{theorem}
\label{curved Hadamard}
Let $\Gamma(t)$ be a family of moving interfaces in $\Omega$. Assume $\vc z\in W^{1,\infty}(\Omega;\mathbb{R}^3)$ is such that the function $\vc Z=\vc Z(\vc x,t)$ satisfying 
%\nabla \vc Z(\vc x,t)= 
\beq
\label{Z cond 2}
\begin{cases}
\nabla \vc z(\vc x), \quad &\text{a.e. in }\Omega_M(t)\\
\mt 0, \quad &\text{a.e. in }\Omega_A(t)\\
\end{cases}
\end{equation}
with $\Omega_A(t),\Omega_M(t)$ as in Definition {\rm\ref{DEf inter}}, is in $W^{1,\infty}(\Omega;\mathbb{R}^3)$ for every $t\in I_T$. 
Then, there exist $\vc a\in L^{\infty}(\Omega;\mathbb{R}^3)$, $\vc n \in L^{\infty}(\Omega;\mathbb{S}^2)$ such that
\begin{equation}
\label{r1 curve}
\nabla \vc z(\vc x)=\vc a(\vc x)\otimes \vc n(\vc x),\qquad\text{a.e. }\vc x\in \Omega.
\end{equation}
Conversely, let $\vc z\in W^{1,\infty}(\Omega;\mathbb{R}^3)$ be such that \eqref{r1 curve} is satisfied and $\vc z(\Omega)$ is contained in the image of an %continuous simple open curve $ \vc c\colon I_T\to\R^3$ which is $1$-rectifiable.
absolutely continuous simple curve $ \vc c\colon I_T\to\R^3$ of finite length. Then, if $|\vc a|>0$ a.e. in $\Omega$, there exists a family of moving interfaces in $\Omega$, and a $\vc Z=\vc Z(\vc x,t)$ in $W^{1,\infty}(\Omega;\mathbb{R}^3)$ satisfying \eqref{Z cond 2} for every $t\in I_T$. Furthermore, $\mathscr{L}^3 \bigl( \Omega\setminus(\Omega_A(t)\cup\Omega_M(t))\bigr)=0$ for every $t\in I_T$.
\end{theorem}
{
\begin{remark}
\rm{The assumptions on the image of $\vc z$ in Theorem \ref{curved Hadamard} are motivated by the following observation: if $\vc z \in C^1(\Omega;\R^3)$, and $\nabla\vc z$ is rank-one everywhere in $\Omega$, then the constant rank theorem implies that, around every $\vc x\in\Omega$, the image of $\vc z$ is a simple absolutely continuous curve. However, the set $\vc z(\Omega)$ can a priori show branching and other complex structures even in the regular case (e.g., if $\Omega$ is non-convex). For the sake of clarity of the proof, in this paper we restrict ourselves to the easier case where $\vc z(\Omega)$ is a simple absolutely continuous curve.
{Nonetheless, a statement similar to the second implication in Theorem \ref{curved Hadamard} can be proved for maps $\vc z \colon \Omega\to\R^3$ whose image satisfies}
%Nonetheless, it is possible to generalise the second implication in Theorem \ref{curved Hadamard} to
%maps $\vc z \colon \Omega\to\R^3$ whose image satisfies
\begin{itemize}
\item {$\vc z (\Omega)$ is $1-$rectifiable};
\item for $\mathscr H^1-$a.e. $\boldsymbol\xi\in\vc z(\Omega)$ there exist an open ball $\mathcal D_{\boldsymbol\xi}\subset\R^3$ such that $\mathcal D_{\boldsymbol\xi}\cap \vc z(\Omega)$ is a simple curve of finite length which is absolutely continuous. 
\end{itemize}
}
\end{remark}

\begin{remark}
\rm{
Assume that the moving mask assumption holds, and that we can reconstruct $\vc z$ from experimental observations. Then, provided the image of $\vc z$ satisfies the stated assumptions, the second part of Theorem \ref{curved Hadamard} gives a useful tool to reconstruct phase interfaces during the phase transformation.
}
\end{remark}
}
For the proof of Theorem \ref{curved Hadamard} we need the following Lemma:
\begin{lemma}
\label{measurable}
Let $\mt f\in L^\infty(\Omega;\R^{3\times3})$ be such that $\mt f(\vc x)=(\vc b\otimes\vc  m)(\vc x)$ for a.e. $\vc x\in\Omega$. Then, there exist $\vc a,\vc n\in L^\infty(\Omega;\R^3)$ such that $\mt f(\vc x) = \vc a(\vc x)\otimes \vc n(\vc x)$ and $|\vc n(\vc x)|=1$ for a.e. $\vc x\in\Omega$.
\end{lemma}
\begin{proof}
This is just a matter of measurability of $\vc a,\vc n$. As $\mt f$ is measurable, so is $\mt f^T\mt f= (|\vc b|^2 \vc m\otimes\vc  m)$, so is its trace $\tr(\mt f^T\mt f)=|\vc b|^2|\vc m|^2$ and so is the function 
\[
\mt g := \begin{cases}
|\vc b|^{-2}|\vc m|^{-2}\mt f^T\mt f,\qquad &\text{ if $|\vc b|^2|\vc m|^2\neq 0$,}\\
0,\qquad &\text{ otherwise}.
\end{cases}
\]
Therefore, we define $\Omega_1:=\{\vc x\in\Omega\colon g_{11}(\vc x) \neq 0\}$ and
$$\vc n(\vc x) = \bigl((g_{11}(\vc x))^{\frac12},\;g_{12}(\vc x)(g_{11}(\vc x))^{-\frac12},\;g_{13}(\vc x)(g_{11}(\vc x))^{-\frac12}\bigr)^T,$$
for almost every $\vc x\in\Omega_1$. This is actually possible because $g_{ii} \geq 0$ a.e. in $\Omega$, for $i=1,2,3$. Define also
$$
\Omega_2:=\{\vc x\in\Omega\setminus \Omega_1 \colon g_{22} \neq 0\},\qquad \Omega_3:=\{\vc x\in\Omega\setminus (\Omega_1\cup\Omega_2) \colon g_{33} \neq 0\},
$$
and define $n$ in $\Omega_2$, $\Omega_3$ respectively by
\[
\begin{split}
\vc n = \bigl(g_{21}(g_{22})^{-\frac12},\;(g_{22})^{\frac12},\;g_{32}(g_{22})^{-\frac12}\bigr)^T,\\ 
\vc n = \bigl(g_{31}(g_{33})^{-\frac12},\;g_{32}(g_{33})^{-\frac12},\;(g_{33})^{\frac12}\bigr)^T.
\end{split}
\]
Therefore, choosing $\vc n$ arbitrarily and such that $|\vc n|=1$ in the set where $\mt g = \mt 0$, we have constructed $\vc n\in L^\infty(\Omega;\R^3)$ as desired. Defining
$
\vc a := \mt f\vc n,
$
we thus conclude the proof.
\end{proof}

\begin{proof}[Proof of Theorem \ref{curved Hadamard}]
We first prove \eqref{r1 curve}.\\
Let $N_{\vc z}$ be the set where $\vc z$ is not differentiable and remark that, by the hypotheses, $N:=N_{\vc z}\cup (\Omega\setminus\mathcal B)$ is an $\mathscr L^3$-negligible set. Let $\vc x_0\in\Omega\setminus N$ and take $\mathcal U_{\vc x_0}$ to be a neighbourhood of $\vc x_0$ as in Definition \ref{DEf inter} (\ref{def mi 2}). By taking a smaller connected neighbourhood of $\vc x_0$, which we still denote by $\mathcal U_{\vc x_0}$, we can assume that $\mathcal G_{\vc x_0}\cap \mathcal U_{\vc x_0}$ is connected. 
{We first claim that $\vc z$ is constant on $\mathcal{G}_{\vc x_0}\cap \mathcal U_{\vc x_0}$. Indeed, as} %As 
$\nabla \vc Z(\vc x,t^*)=0$ a.e. in $\Omega_A(t^*)$, the continuity of $\vc Z(\vc x,t^*)$ implies that $\vc Z(\vc x,t^*)$ must be constant on every connected component of $\overline{\Omega}_A(t^*)$. Since Definition \ref{DEf inter} (\ref{def mi 2}) implies $\mathcal G_{\vc x_0}\cap \mathcal U_{\vc x_0}\subset \overline\Omega_A(t^*)$, we must have $\vc Z(\cdot,t^*)=\hat{\vc c}$ for some $\hat{\vc c}\in \R^3$ on $\mathcal G_{\vc x_0}\cap \mathcal U_{\vc x_0}$. On the other hand, continuity of $\vc z, \vc Z(\cdot,t^*)$ together with \eqref{Z cond 2} imply that on every connected component of $\overline{\Omega}_M(t^*)$ $\vc z = \vc Z(\cdot,t^*)+\bar{\vc c}$ for some $\bar{\vc c}\in\R^3$ depending on the connected component. Therefore, as by Definition \ref{DEf inter} (\ref{def mi 2}) $\mathcal G_{\vc x_0}\cap \mathcal U_{\vc x_0}\subset \overline\Omega_M(t^*)$, the fact that $\vc Z(\cdot,t^*)$ is constant on $\mathcal G_{\vc x_0}\cap \mathcal U_{\vc x_0}$ implies that so must  be $\vc z$.\\

%
%$\mathcal G_{\vc x_0}\cap\overline\Omega_A(t^*)\cap \mathcal U_{\vc x_0}$ is connected and hence $Z$ must be equal to a constant $\hat{\vc c}\in \R^3$ on $\mathcal G_{\vc x_0}\cap \mathcal U_{\vc x_0}$. 
Now, as $\mathcal G_{\vc x_0}$ is a Lipschitz $2$-graph, we can find a Lipschitz change of coordinates %$\boldsymbol\psi$ such that 
$\boldsymbol \psi \colon\mathcal U_{\vc x_0}\to V$ such that
$$ %\boldsymbol \psi \colon\mathcal U_{\vc x_0}\to V, \qquad \text{such that}\qquad 
\boldsymbol \psi(\mathcal G_{\vc x_0}\cap \mathcal U_{\vc x_0})=\bigl\{ \vc x\in\R^3\colon \vc x\cdot \vc n(\vc x_0)=c_\Gamma\bigr \}\cap V,%\quad \mathscr H^2\text{- a.e.}
$$ for some {open connected} $V\subset\mathbb R^3$, $c_\Gamma\in\mathbb{R}$ and where $\vc n(\vc x_0)$ is the normal vector to $\mathcal G_{\vc x_0}$ at ${\vc x_0}$ pointing outwards from $\Omega_M(t^*)$. Let us denote $\boldsymbol \psi(\vc x)=\bar{\vc x}$ for every $\vc x\in \mathcal{U}_{\vc x_0}$. We define $\bar{\vc z}$ as $\bar{\vc z}(\bar{\vc x})=\vc z(\boldsymbol \psi^{-1}(\bar{\vc x}))$, and assuming without loss of generality that $\vc n(\vc x)=\vc e_3$, we get that
$$
\bar{\vc z}(\bar {\vc x}_0)=\bar{\vc z}(\bar{\vc  x}_0+s\vc e_i)=\hat{\vc  c},\qquad i=1,2, %|s|<r,\,i=1,2\qquad \mathscr H^2\text{- a.e.}.
$$
for each $s$ such that $\bar{\vc  x}_0+s\vc e_i\in V.$ This is due to the fact that $\vc z(\vc x)=\hat {\vc c}$ for every $\vc x\in \mathcal U_{\vc x_0}\cap\mathcal{G}_{\vc x_0}$. 
Therefore,
$$
\frac{\partial \bar{\vc  z}}{\partial \bar{x}_i}(\bar{\vc  x})=0,\qquad i=1,2.
$$
On the other hand, as $\vc x_0\in\Omega\setminus N$, we have
$$
\nabla_\vc x \vc z(\vc x_0)=\nabla_{\bar{\vc x}}\bar {\vc z}(\boldsymbol\psi(\vc x_0))\nabla_\vc x\boldsymbol\psi(\vc x_0).
$$
Since $\vc x_0$ is a regular point, $\boldsymbol\psi$ can be chosen to be differentiable in $\vc x_0$, and therefore $\bar{ \vc x}_0$ is a point of differentiability for $\bar{\vc z}$. Therefore, there exists $\vc a\in\mathbb R^3$ such that
$$
\nabla_{\bar{\vc x}}\bar{\vc z}(\bar{\vc x}_0)=\vc a\otimes \vc n(\vc x_0).
$$
By putting together the last two identities and using the fact that $N$ is negligible we finally deduce \eqref{r1 curve}. Measurability of $\vc a,\vc n$ follows from Lemma \ref{measurable}. \\

We now prove the second statement. %Let us assume without loss of generality that $\vc c$ is parametrized by its arc-length. 
We first remark that since $\vc c$ is absolutely continuous and of finite length, it belongs also to $W^{1,1}(I_T;\R^3)$ and there exists $\vc d\in W^{1,\infty}(I_T^*;\R^3)$ for some interval $I^*_T\subset\R$ such that $\vc c(I_T)=\vc d(I^*_T)$ (see e.g., \cite{ABC} and references therein). Therefore $\vc c(I_T)$ is $1$-rectifiable.
By Theorem \ref{Fed thm}, $\Gamma(t):=\vc z^{-1}(\vc c(t))\cap\Omega$ are $2$-rectifiable surfaces for almost every $t$, and $\vc z$ is equal to a constant on them. Defining
\beq%gin{align*}
\label{omeghi}
\begin{split}
\Omega_M(t) &:=\bigl\{\vc x\in\Omega\colon  \exists s\in[0,t) \text{ such that } \vc x\in \vc z^{-1}(\vc c(s)) \bigr\},\\
 \Omega_A(t)&:=\bigl\{\vc x\in\Omega\colon \vc x\notin \vc z^{-1}(\vc c(s)),\, \forall s\in[0,t] \bigr\},
\end{split}
\eeq%nd{align*} 
it is easy to see that Definition \ref{DEf inter}\eqref{def i} is satisfied, {provided we can show that} $\Omega_A(t),\Omega_M(t)$ are open. To this end, let us fix $t^*\in I_T$, $\hat{\vc x}\in\Omega_M(t^*)$, and let us denote by $s_{\hat{\vc x}}\in [0,t^*)$ the point such that $\vc z(\hat{\vc x}) = \vc c(s_{\hat{\vc x}})$. %The continuity of $\vc c$ implies that fixed $\eps = \frac12|\vc c(t)-\vc c(s_{\hat{\vc x}})|$, there exists $\delta>0$ such that for every $r\in I_T$ satisfying $|r-s_{\hat{\vc x}}|\leq\delta$ it must hold $|\vc c(r)-\vc c(s_{\hat{\vc x}})|\leq \eps $. We remark that $\delta<|t-s_{\hat{\vc x}}|$. On the other hand, a
As $\vc z$ is Lipschitz, %for every $\hat{\vc x}\in\Omega_M(t)$, 
we can define $R:=\frac12 \|\nabla \vc z\|_{L^{\infty}}^{-1}|\vc c(t^*)-\vc z(\hat {\vc x})|$, so that
%there exists an open neighbourhood $\mathcal U$ of $\bar x$ such that
\beq
\label{a contrad}
|\vc z(\vc x)-\vc z(\hat {\vc x})|\leq \|\nabla \vc z\|_{L^{\infty}} |\vc x- \hat {\vc x}|\leq\frac12 |\vc c(t^*)-\vc z(\hat {\vc x})| = \frac12|\vc c(t^*)-\vc c(s_{\hat{\vc x}})|, %\qquad \forall \vc x\in B_R(\hat{\vc x}).
\eeq
%Let $s_{\hat{\vc x}},s_\vc{x}$ be such that $\vc z(\hat{\vc x}) = \vc c(s_{\hat{\vc x}}) $ and $\vc z({\vc x}) = \vc c(s_{{\vc x}}) $.
for all $\vc x\in B_R(\hat{\vc x})$. Suppose now that in $B_R(\hat{\vc x})$ there exists a point $\vc x_0$ such that $\vc z(\vc x_0)=\vc c(t_0)$ for some $t_0\geq t$. Then the segment connecting $\vc x_0$ to $\hat{\vc x}$ is still contained in $B_R(\hat{\vc x})$, and its image through $\vc z$ must be a connected part of the image of $\vc c$. But as $\vc c$ is a simple curve, this implies that there exists $\vc x_1\in B_R(\hat{\vc x})$ such that $\vc z(\vc x_1)=\vc c(t^*)$, which is in contradiction with \eqref{a contrad}. 
Therefore, for every $t^*\in I_T$, $\hat{\vc x}\in\Omega_M(t^*)$ there exists an open ball centred at $\hat {\vc x}$ contained in $\Omega_M(t^*)$, and therefore $\Omega_M(t^*)$ is open. The same argument can be used to show that also $\Omega_A(t)$ is open for each $t$. Clearly, $\Gamma(t)$ is sequentially closed in $\Omega$ and $\Omega_M(t)\cup\Gamma(t), \Omega_A(t)\cup\Gamma(t)$ are closed in $\Omega$ as well. In this way we have also shown that $\vc Z(x,t)$ defined as
$$
\vc Z(x,t)= 
\begin{cases}
\vc z(x), \quad &\text{in }\Omega_M(t)\\
\vc c(t), \quad &\text{in }\Omega_A(t)\\
\end{cases}
$$
is in $W^{1,\infty}(\Omega;\mathbb{R}^3)$ for every $t\in I_T$. \\

Now, since $\Gamma(t)$ is $2$-rectifiable for almost every $t$, in order to show that $$\mathscr{L}^3 \bigl( \Omega\setminus(\Omega_A(t)\cup\Omega_M(t))\bigr)=0$$ for every $t\in I_T$ it is sufficient to prove that
%s 
$$
\mathcal C:=\bigl \{\vc x\in\Omega\colon \vc x\in \Gamma(t),\,t\in I_T,\,\Gamma(t) \text{ is not $2$-rectifiable}\bigr\} 
$$
%have
has null $\mathscr L^3$ measure. By Theorem \ref{Fed thm}, $\mathcal C$ is the preimage through $\vc z$, which is continuous, of a set of measure zero, and is hence measurable. %while 
Now, we notice that by choosing $g$ to be the indicator function on $\mathcal C$ in the coarea formula \eqref{coarea}, and identifying $\mathcal{Z}$ with the support of $\vc c$, we have
\beq
\label{coarea use}
0 \leq \mint g(\vc x) |\vc a|\, \mathrm{d}\vc x
=\int_{\mathcal Z}\int_{\vc z^{-1}(\boldsymbol\xi)}g(\vc s)\,\mathrm{d}\mathscr H^2(\vc s)\,\mathrm{d}\mathscr H^1(\boldsymbol\xi)=0
\eeq
as, by Theorem \ref{Fed thm}, this can just happen for a set of measure zero in $\mathcal Z$. % and thanks to the usual convention in measure theory that $0\cdot\infty=0$. 
This, together with the fact that $|\vc a|>0$ a.e., leads to $\mathscr L^3(\mathcal{C})=0$.\\ 

The rest of the proof is devoted to prove that Definition \ref{DEf inter} \eqref{def mi 2} is satisfied. To this aim, we first claim that for every point $\vc x\in \mathcal D$, with
$$
\mathcal D:=\bigl \{\vc x\in\Omega\colon \nabla\vc z(\vc x)\text{ exists, and }\nabla\vc z(\vc x)\neq\mt0\bigr\} ,
$$
there exist a Lipschitz $2$-graph $\mathcal G_\vc x$ which is differentiable at $\vc x$, an open neighbourhood $\mathcal U_\vc x$  and a $t^*\in I_T$ satisfying $\mathcal G_\vc x\cap \mathcal U_\vc x \subset \Gamma(t^*)\cap \mathcal U_\vc x$. % is a subset of $\Omega\setminus\mathcal D$. %, with
%$$
%\mathcal D:=\bigl \{\vc x\in\Omega\colon \nabla\vc z(\vc x)\text{ exists, and }\nabla\vc z(\vc x)\neq\mt0\bigr\} ,
%$$
%%the points where $\vc z$ is not differentiable, or $\nabla \vc z=\mt 0$, 
%and therefore of measure zero. 
In order to do that, we would need a generalised version of the constant rank theorem. However we were not able to find a version of it in the literature suitable to our application. We hence strongly exploit the structure of the image of $\vc{z}$ and a weak version of the implicit function theorem. Here and below, given a vector $\vc v\in\R^3$, we denote by $v_i$ its $i-$th component. Let us consider a generic $\hat{\vc x}\in \mathcal D$ and suppose, without loss of generality, that $a_1(\hat{\vc x})\neq 0$ and that $\vc n(\hat{\vc x})=\vc e_3$. In this case, a version of the implicit function theorem as the one in \cite[Thm. E]{Halkin} gives the existence of a connected neighbourhood $\mathcal N$ of $(\hat{x}_1,\hat{x}_2)$, and of a function $\psi\colon\mathcal{N}\to\R$, such that $\psi(\hat{x}_1,\hat{x}_2)=\hat{x}_3$, and $z_1({x}_1,{x}_2,\psi({x}_1,{x}_2))=z_1(\hat{\vc x})$ for every $(x_1,x_2)\in \mathcal{N}$. Furthermore $\psi$ is differentiable in $(\hat{x}_1,\hat{x}_2)$ and hence continuous and Lipschitz in $\mathcal N$, and $\nabla \psi(\hat{x}_1,\hat{x}_2) = \vc 0$. \\
Fixed $\eps = \frac{|a_1(\hat{\vc x})|}2$, the fact that $\vc z$ is differentiable in $\hat{\vc x}$ implies the existence of $\delta>0$ such that
$$
z_1(\hat{\vc x}+\rho \vc e_3)-z_1(\hat{\vc x})=\rho a_1+r\rho,\qquad\forall |\rho|<\delta,
$$
and where $|r|<\eps$. Therefore, 
\beq
\label{diverso sempre}
\begin{split}
z_1(\hat{\vc x}+\rho \vc e_3)>z_1(\hat{\vc x}),\quad\text{if $a_1(\hat{\vc x})\delta>a_1(\hat{\vc x})\rho>0$},\\%\qquad\text{and}\qquad
z_1(\hat{\vc x}+\rho \vc e_3)<z_1(\hat{\vc x}),\quad\text{if $-a_1(\hat{\vc x})\delta<a_1(\hat{\vc x})\rho<0$},
\end{split}
\eeq
for all $|\rho|<\delta$. This implies the existence of $h>0$ and $\vc c(t^*+h),\vc c(t^*-h)$ in $\vc z(\Omega)$ such that $$c_1(t^*+h)>c_1(t^*)>c_1(t^*-h).$$ Here, $t^*\in I_T$ is such that $\vc z(\hat{\vc x}) = \vc c(t^*)$. Furthermore, since $\vc z$ is Lipschitz, the dependence of $c(t(\vc x)) := \vc z(\vc x)$ is continuous. This together with \eqref{diverso sempre} and the fact that $\vc c$ is simple, imply that the unique path connecting $\vc c(t^*\pm h)$ to $\vc c(t^*)$ must be such that $c_1(t^*+s)>c_1(t^*)>c_1(t^*-s)$ either for every $s\in(0,h)$ or for every $s\in(-h,0)$. 
Suppose now the existence of $(x_1,x_2)\in\mathcal{N}$ such that $\vc z(x_1,x_2,\psi(x_1,x_2)) \neq \vc c(t^*)$. By continuity of $c(t(\vc x))$ there exist $(\tilde x_1,\tilde x_2)\in\mathcal{N}$ such that $\vc z(\tilde x_1,\tilde x_2,\psi(\tilde x_1,\tilde x_2)) = \vc c(t^*+s)$ for some $s$ with $0<|s|<h$. Thus, at the same time we should have $c_1(t^*+s)=c_1(t^*)$ because we are on a level set for $z_1$, and $c_1(t^*+s)\neq c_1(t^*)$, which leads to a contradiction. %Therefore, since $\vc c$ is simple and open, the unique path connecting $\vc c(t^*\pm h)$ to $\vc c(t^*)$ must be such that $c_1(t^*+s)>c_1(t^*)>c_1(t^*-s)$ either for every $s\in(0,h)$ or for every $s\in(-h,0)$. 
%
%
%Here, as above, $t^*\in I_T$ is such that $\vc z(\hat{\vc x}) = \vc c(t^*)$. The fact that $\vc c$ is simple and open thus implies that $c_1$ constant implies $\vc c$ constant.
%%$\vc c(t^*+h),\vc c(t^*-h)$ in $\vc z(\Omega)$ such that, given that $\vc c$ is simple and open, 
We hence showed that $c_1$ is constant implies also that $c_2,c_3$ are constants, that is $\vc z({x}_1,{x}_2,\psi({x}_1,{x}_2))=\vc z(\hat{\vc x})=\vc c(t^*)$ for every $(x_1,x_2)\in \mathcal{N}$. This concludes the proof of the claim.%which concludes the proof of the claim. %Suppose not, then there exists a continuous path along the curve $\vc c$ connecting $\vc c(t^*)$ to 
\\

It remains to prove that for all $\vc x\in\mathcal D$, it holds $\Gamma(t^*)\cap \mathcal U_\vc x\subset \overline\Omega_{L}(t^*)\cap\overline\Omega_{R}(t^*)\cap\mathcal{U}_\vc x$, where again $t^*	\in I_T$ is such that $\vc x\in\Gamma(t^*).$
Suppose first that there exists a neighbourhood $\mathcal U$ of ${\vc x_0}\in \Gamma(t^*)$ for some $t^*\in I_T$ such that 
\beq
\label{sat 11}
\Omega_A(t^*)\cap\mathcal U=\varnothing\qquad \text{ or }\qquad \Omega_M(t^*)\cap\mathcal U=\varnothing.
\eeq 
This is $\vc z(\vc x_0)=\vc c(t^*)$, and $\vc z(\vc x) =\vc c(s(\vc x))$ with $s(\vc x)>t^*$ or $s(\vc x)<t^*$ for every $\vc x \in \mathcal{U}.$ We want to prove that either $\vc z$ is not differentiable in $\vc x_0$, or $\nabla\vc z(\vc x_0)=\mt0$. Suppose not, then there exists $\beta_j\in\R\setminus\{0\}$ and a unit vector $\vc v_j$ %choose three orthonormal basis $\{e^j_1,e^j_2,e^j_3\}$ centred at $x$ 
such that $\nabla z_j(\vc x_0)\cdot \vc v_j=\beta_j$ for some $j=1,2,3$. Observe also that the differentiability of $\vc z$ implies the existence of $\delta_j\in(0,1)$ such that
\beq
\label{z_j pos}
z_j(\vc x_0+\alpha \vc v_j)- z_j(\vc x_0)-\alpha\beta_j=\alpha r_j(\alpha \vc v_j), \qquad\forall\alpha\colon |\alpha|<\delta_j,
\eeq
for some continuous functions $r_j$ bounded in modulus by $\frac{\beta_j}{2}$. %For every $j=1:3$ such that $\beta_j\neq0$, \eqref{z_j pos}
This implies that $ z_j(\vc x_0+\alpha \vc v_j)- z_j(\vc x_0)$ has the same sign as $\alpha\beta_j$. Therefore, as $\vc c$ is simple, there exists an interval $(t_{\delta_j},t^*)$ (or $(t^*,t_{\delta_j})$) where $c_j(t)-c_j(t^*)$ is both strictly positive and strictly negative for every $t\in(t_{\delta_j},t^*)$ (or in $(t^*,t_{\delta_j})$), thus leading to a contradiction. Therefore, if $\vc z$ is differentiable at $\vc x_0\in\Gamma(t^*)$ and $\nabla\vc z(\vc x_0)\neq \mt0$, then $\vc x_0\in\overline\Omega_A(t^*)\cap\overline\Omega_M(t^*)$. 

Now, by the coarea formula \eqref{coarea} with $g$ chosen to be the characteristic function of $\mathcal D$, we notice that
$$
0 = \int_{\mathcal Z}\int_{f^{-1}(\boldsymbol\xi)\cap\Omega}g(\vc s)\,\mathrm{d}\mathscr H^2(\vc s)\,\mathrm{d}\mathscr H^1(\boldsymbol\xi),
$$
and we deduce that $\vc z$ is differentiable with $\nabla\vc z\neq\mt0$ for $\mathscr H^2$-almost every $\vc x\in\Gamma(t)$ for almost every $t\in I_T$. Let us call $J_T$ the subset of $I_T$ such that $\vc x\in \mathcal D$ $\mathscr H^2$-almost everywhere in $\Gamma(t)$ for every $t\in J_T$. By arguing as above for the set $\mathcal C$, the coarea formula implies that the set of $\vc x\in\Omega$ such that $\vc z(\vc x)\in \vc c(I_T\setminus J_T)$ has measure zero. 
We can hence focus without loss of generality on $\Gamma(t^*)$ for some $t^*\in J_T$. 
Suppose now that there does not exist a neighbourhood of $\vc x_s\in\Gamma(t^*)$ such that $\Gamma(t^*)\cap\mathcal U\subset\overline\Omega_A(t^*)\cap\overline\Omega_M(t^*)$. In this case, as $\Gamma(t^*)$ is closed in $\Omega$, there exists %a point $\vc x_s$ in $\Gamma(t^*)$ and 
a neighbourhood $\mathcal U_s$ of $\vc x_s$ satisfying \eqref{sat 11}. 
However, as $\nabla\vc z$ exists and is non null $\mathscr H^2$ almost everywhere on $\Gamma(t^*)$, there exists $\vc x_a\in\Gamma(t^*)\cap\mathcal U_s$, $\vc x_a\in\mathcal{D}$, and which must hence be in $\overline\Omega_A(t^*)\cap\overline\Omega_M(t^*)$, thus leading to a contradiction. We have therefore proved that the constructed family of moving interfaces $\Gamma(t)$ satisfies the condition in Definition \ref{DEf inter} (\ref{def mi 2}), which concludes the proof.
%and therefore $\Gamma$ is a family of moving interfaces in $\Omega$.
\end{proof}

The following corollaries are straightforward consequences of the above theorem:
{
\begin{corollary}
\label{reg mov mask}
Let $\vc y\in W^{1,\infty}(\Omega;\R^3)$ satisfy a regular moving mask approximation. Then, there exist $\vc a \in L^\infty(\Omega;\R^3),\vc n\in L^\infty(\Omega;\mathbb{S}^2)$ such that
\beq
\label{mov mask}
\nabla \vc y = \mt Q + \vc a(\vc x)\otimes\vc n(\vc x),\qquad \text{a.e. $\vc x\in \Omega$.}
\eeq
Conversely, if $\vc y\in W^{1,\infty}(\Omega;\R^3)$ satisfies \eqref{mov mask} for some $\mt Q\in SO(3),\vc a \in L^\infty(\Omega;\R^3),\vc n\in L^\infty(\Omega;\mathbb{S}^2)$ and $\vc z(\Omega)$, with $\vc z(\vc x)=\vc y(\vc x)-\mt Q\vc x$, is contained in an absolutely continuous simple curve of finite length, then $\vc y$ satisfies a regular moving mask approximation. 
\end{corollary}
}
\begin{corollary}
\label{Coroll 1}
Let $T>0$, $\Gamma(t)$ be a family of moving interfaces and $\overline{\Omega}_A(t),$ $\overline\Omega_M(t)$ be connected for every $t\in(0,T)$. Then, $\vc Z\in W^{1,\infty}(\Omega,\R^3)$ for each $t\in (0,T)$ satisfying \eqref{Z cond 2} is equal to 
$$
\vc Z(x,t)= 
\begin{cases}
\vc z(x)+\vc c_1(t), \quad &\text{in }\Omega_M(t),\\
\vc c_2(t), \quad &\text{in }\Omega_A(t),
\end{cases}
$$
for some $\vc z\in W^{1,\infty}(\Omega,\R^3)$ such that $\nabla \vc z =\vc a\otimes\vc  n$ almost everywhere. %Furthermore, if $(c_1-c_2)\colon (0,T)\to\R^3$ is a simple curve \beq
%\label{blablabla} \Gamma(t)=\bigl\{x\in\Omega\colon z(x)+c_1(t)=c_2(t)\bigr\}.\eeq
\end{corollary}
\begin{proof}
The statement follows directly from Theorem \ref{curved Hadamard}, the fact that $\vc Z(\vc x,t)$ is continuous in $\vc x$ for each $t$ and the hypothesis that $\overline\Omega_A(t), \overline\Omega_M(t)$ are connected. %The second statement follows again by the continuity of $Z(x,t)$ in $x$ for each $t$, the fact that $(c_1-c_2)(t)$ is a simple curve, and the fact that $z(\Omega)=(c_1-c_2)(I_T)$
\end{proof}

{Corollary \ref{Coroll 1} hence implies that, under the above hypotheses, for each $t\in I_T$, the interface $\Gamma(t)$ is a subset of $\{\vc x\in\Omega\colon \vc z(\vc x)=\vc c_2(t)-\vc c_1(t)\} 
$, that is of a level set of $\vc z.$ This also means that the image of $\vc z$ is a one-dimensional curve. 
However $\Gamma(t)$ does not need to coincide with the family of moving interfaces constructed in the proof of Theorem \ref{curved Hadamard}, even if $\vc c_2-\vc c_1$ is absolutely continuous, simple and of finite length. 
Indeed, in the proof of Theorem \ref{curved Hadamard} the phase interfaces must be constructed as subsets of level sets for $\vc z$, but the construction of $\Omega_A(t),\Omega_M(t)$ is arbitrary and could be done differently. 
For example one could swap the definition of $\Omega_A(t)$ and $\Omega_M(t)$ in \eqref{omeghi}, or replace $s\in[0,t)$ and $s\in[0,t]$ in the definition of $\Omega_M$ and $\Omega_A$ respectively with $|s-t_0|< t$ and $s\in I_T\setminus [t_0-t,t_0+t]$ for some $t_0\in I_T$, thus getting a different family of moving interfaces.\\}
\indent 
The next corollary gives some information about the interface velocity. We define the normal velocity of $\Gamma(t^*)$ at time $t^*\in I_T$ and at $\vc x\in\Gamma(t^*)$, namely $(\vc v\cdot \vc n)(\vc x,t^*)$, as $\dot{\boldsymbol \gamma}(t^*)\cdot \vc n(\vc x,t^*)$, where $\vc n(\vc x,t^*)$ is the unit normal to $\Gamma(t^*)$ at $\vc x\in\Gamma(t^*)$, and $\boldsymbol \gamma(t)$ is a generic absolutely continuous path differentiable at $t^*$ such that $\boldsymbol \gamma(t)\in\Gamma(t)$ for each $t\in I_T$ and $\boldsymbol \gamma(t^*)=\vc x$. Clearly $(\vc v\cdot \vc n)(\vc x,t^*)$ is well defined if its value is independent of the choice of $\boldsymbol \gamma$ among the admissible paths, and if $\vc n(\vc x,t^*)$ is well defined. 

\begin{corollary}
\label{Coroll 2}
Let $\vc z\in W^{1,\infty}(\Omega,\R^3)$ satisfy \eqref{r1 curve}, $|a|>0$ a.e. in $\Omega$, and $\vc z(\Omega)$ be contained in the image of an absolutely continuous simple curve $\vc c:I_T\to\R^3$ of finite length. Assume further that $\Gamma(t)$ is the family of moving interfaces constructed in the proof of Theorem {\rm\ref{curved Hadamard}}. Then, the normal velocity of $\Gamma(t)$ at a point $\vc x\in \Gamma(t)$, denoted $(\vc v\cdot \vc n)(\vc x,t)$, satisfies
%
%the hypotheses of Corollary \ref{Coroll 1} hold, and let $(\vc v\cdot \vc n)(\vc x,t)$ be the normal velocity of $\Gamma(t)$ at a point $\vc x\in \Gamma(t)$. Assume also that $|\vc a|>0$ a.e. in $\Omega$ and that $\vc z(\Omega)$ is an absolutely continuous simple curve $\vc c:I_T\to\R^3$. Then,
\beq
\label{velocita}
\vc a(\vc x)(\vc v\cdot \vc n)(\vc x,t)=\dot{\vc  c}(t),{%\color{red}
\qquad\text{a.e. }t\in (0,T),\,\mathscr H^2\text{-a.e. }\vc x\in\Gamma(t).}%\Omega.}
\eeq
\end{corollary}
\begin{proof}
By the coarea formula \eqref{coarea} with $g$ chosen to be the characteristic function of the set where $z$ is not differentiable and $|a|>0$, we notice that
$$
0 = \Tint\int_{\vc z^{-1}(\vc c(t))\cap\Omega}g(\vc s)\,\mathrm{d}\mathscr H^2(\vc s)|\dot {\vc c}(t)|\,\mathrm{d}t.
$$
As the argument in the integral is non negative, we deduce that $ \vc z$ is differentiable and $|\vc a|>0$ for $\mathscr H^2$-almost every $\vc x\in\Gamma(t)$ for almost every $t\in I_T$. As showed in the proof of Theorem \ref{curved Hadamard} $\vc n$ is well defined for all these $\vc x$. Let us consider $\boldsymbol \gamma(t)$, an absolutely continuous path in $\Omega$ such that $\boldsymbol \gamma(t) \in\Gamma(t)$ for every $t\in I_T$. We have
$$
\vc z(\boldsymbol \gamma(t))=\vc c(t),\qquad \forall t\in(t_0,t_1),
$$
for some $0\leq t_0<t_1\leq T$. {%\color{red}
Taking the time derivative of this identity we get
$$
\dot{\vc  c}(t)=\nabla \vc z(\boldsymbol \gamma(t))\dot {\boldsymbol \gamma}(t) = \vc a(\boldsymbol \gamma(t))\bigl( \dot {\boldsymbol \gamma}(t)\cdot \vc n(\boldsymbol \gamma(t),t)\bigr)=
\vc a(\vc x)\bigl(\dot {\boldsymbol \gamma}(t)\cdot \vc n(\vc x,t)\bigr),
$$
%for almost every $\vc x(t)\in \Omega$, % and almost every $t$. Finally, by coarea formula \eqref{coarea} with $g$ chosen to be the indicator function on the set where $z$ is not differentiable, we notice that
%$$
%0 = \Tint\int_{z^{-1}(c(t))\cap\Omega}g(s)\mathrm{d}\mathscr H^2(s)|\dot c(t)|\mathrm{d}t,
%$$
%and as the argument in the integral is non negative, we deduce that $z$ is differentiable for $\mathscr H^2$-almost every $x\in\Gamma(t)$ for almost every $t\in I_T$, 
which is the claimed result, as $\bigl(\dot {\boldsymbol \gamma}(t)\cdot \vc n(\vc x,t)\bigr)$ is independent of $\gamma$ chosen for a.e. $t\in I_T$, a.e. $\vc x\in \Gamma(t).$}
\end{proof}

\begin{remark}
\normalfont
An important consequence of the above corollary is that $\frac{\vc a(\vc x)}{|\vc a(\vc x)|}$ is constant $\mathscr H^2$ almost everywhere on $\Gamma(t)$ for almost every $t$. At the same time, there might be jumps in $|\vc a(\vc x)|$ along a single interface and jumps for $\frac{\vc a(\vc x)}{|\vc a(\vc x)|}$ across interfaces. %Unfortunately, under the hypotheses of Corollary \ref{Coroll 2} we are not able to capture austenite to martensite interfaces which are partially moving and partially not. %However, an interesting consequence of this result in terms of modelling is that every constitutive relation for $\vc v\cdot \vc n$ that closes Problem \eqref{Pb 1} must depend on $\nabla \vc y=\mt 1+\vc a\otimes\vc n$ through $\vc a(\vc x)$. More precisely, in this case, the only dependence in space along the interface is due to $\vc a(\vc x)$. 
\end{remark}

\begin{remark}
\normalfont
If we assume the determinant of $\nabla \vc y=\mt 1+\nabla \vc z=\mt{1}+\vc a\otimes \vc n$ to be a positive constant $\mathfrak{D}$ almost everywhere in $\Omega$, than we can deduce that on almost all interfaces $|\vc a(\vc x)|$ can jump if and only if there is a jump in $\vc n(\vc x)$. Indeed, this is a direct consequence of the following two facts: the first is that the direction of $\vc a(\vc x)$ is fixed on almost all interfaces, the second is that, $\det(\nabla\vc y) = \mathfrak{D}$ a.e. in $\Omega$ implies  $\vc a\cdot \vc n=\mathfrak{D}-1$ a.e. in $\Omega$, and hence, by the coarea formula (see the argument in the proof of Corollary \ref{Coroll 2}), $\vc a\cdot \vc n=\mathfrak{D}-1$, $\mathscr H^2$-almost everywhere on $\Gamma(t)$ for almost all $t$. 
\end{remark}

A different perspective on the velocity of $\Gamma(t)$ is given by
\begin{corollary}
\label{derivata materiale coroll}
Let $\vc z\in W^{1,\infty}(\Omega,\R^3)$ satisfy \eqref{r1 curve}, $|\vc a|>0$ a.e. in $\Omega$, and $\vc z(\Omega)$ be contained in the image of an absolutely continuous simple curve $\vc c:I_T\to\R^3$ of finite length. Assume further that $\Gamma(t)$ is the family of moving interfaces constructed in the proof of Theorem {\rm\ref{curved Hadamard}}. Then,
$$
\langle\dot{\chi}_{\Omega_M},\xi\rangle = |\dot {\vc c}(t)| \int_{\Gamma(t)}\frac{\xi(\vc s)}{|\vc a(\vc s)|}\,\mathrm{d}\mathscr H^2(\vc s),\qquad\forall\xi\in C^0(\Omega), \text{ a.e. $t\in(0,T)$}.
$$
\end{corollary}
\begin{proof}
We first notice that, by the coarea formula,
\[
\begin{split}
\mint\chi_{\Omega_M}(\vc x,t)\xi(\vc x)\,\mathrm{d}\vc x = \int_0^T\int_{\vc z^{-1}(\vc c(\tau))\cap\Omega_A(t)}\frac{\xi(\vc s)}{|\vc a(\vc s)|}\,\mathrm{d}\mathscr H^2(\vc s)|\dot {\vc c}(\tau)|\,\mathrm{d}\tau \\= 
\int_0^t\int_{\vc z^{-1}(\vc c(\tau))\cap\Omega}\frac{\xi(\vc s)}{|\vc a(\vc s)|}\,\mathrm{d}\mathscr H^2(\vc s)|\dot {\vc c}(\tau)|\,\mathrm{d}\tau.
\end{split}
\]
Therefore,
$$
\dert \mint \chi_{\Omega_M}(\vc x,t)\xi(\vc x)\,\mathrm{d}\vc x =  |\dot {\vc c}(t)| \int_{\vc z^{-1}(\vc c(t))\cap\Omega}\frac{\xi(\vc s)}{|\vc a(\vc s)|}\,\mathrm{d}\mathscr H^2(\vc s).
$$
which is the claimed result.

\end{proof}

%%%%%%%%%%%%%%%%%%%%%%%%%%%%%%%%%%%%%%%%%%%%%%%%%

%%%%%%%%%%%%%%%%%%%%%%%%%%%%%%%%%%%%%%%%%%%%%
%%%%%%%%%%%%%%%%%%%%%%%%%%%%%%%%%%%%%%%%%%%%%%
%%%%%%%%%%%%%%%%%%%%%%%%%%%%%%%%%%%%%%%%%%%%
\section{Basic properties of microstructures}
\label{Basic}
According to the results of the previous sections, we can restrict our attention to deformation 
%all assumptions and results stated above, allow to restrict our attention to deformation 
gradients satisfying for every $t\in(0,T)$
%\beq
%\tag{H1}
\[
\begin{cases} \nabla \vc y( \vc x,t) = \mt 1 + \vc a(\vc x)\otimes \vc n(\vc x),\qquad  &\text {a.e. }\vc x\in\Omega_M(t), \\
\nabla \vc y(\vc x,t) = \mt 1,\qquad  &\text{a.e. }\vc x\in\Omega_A(t),
\end{cases}
\]%\eeq
for some $\vc a(\vc x)\in L^\infty(\Omega,\R^3),\,\vc n(\vc x)\in L^\infty(\Omega,\mathbb{S}^2)$ such that $\vc n(\vc x)$ is normal to the austenite-martensite interface in $\vc x$ at a certain time $t^*\in(0,t)$. Here we assumed without loss of generality to have $\mt Q=\mt 1$ in Definition \ref{defi mi}. In this way the martensitic macroscopic deformation gradient is a function of the moving, possibly curved, austenite-martensite interface during phase transition. In light of the above considerations, we assume that martensitic microstructures arising from austenite to martensite phase transitions are described by deformation gradients $\nabla \vc y$ of the form
\beq
\tag{H1}
\label{h1}
\nabla \vc y(\vc x) = \mt 1 + \vc a(\vc x)\otimes \vc n(\vc x),\qquad\nabla \vc y \in K^{qc},\qquad  \text {a.e. $\vc x\in\Omega.$}
\eeq
As the determinant is constant in $K$, it follows that
\beq
\tag{H2}
\label{h2}
\det \nabla \vc y (\vc x) = \mathfrak{D}\qquad\text {a.e. $\vc x\in\Omega,$}
\eeq
for some constant $\mathfrak{D}>0$. \eqref{h1} and \eqref{h2} imply
\begin{equation}
\label{H2 det}
\vc a(\vc x)\cdot \vc n(\vc x) = \mathfrak{D} - 1,
\eeq
for a.e. $\vc x\in\Omega$, and %Here we are also assuming that $\mathfrak{D}\neq 1$, which is physically reasonable for martensite. Furthermore, $\vc a$ and $\vc n$ must satisfy some other relations, which are direct consequences of (H1) and (H2). For example, as $\nabla \vc y$ is a gradient, each row must be curl free, i.e.,
\begin{equation}
\label{H3}
\curl(a_i(\vc x) \vc n(\vc x))=\vc 0\qquad\text{ for each } i=1,2,3,
\end{equation}
in a weak sense. \\

In conclusion, in what follows we look at martensite microstructures as a part of the domain where (H1) and (H2) hold. 
We first begin with an estimate for the norm of $\vc a(\vc x)$:
 \begin{proposition} Let $\lambda_{max}$ and $\lambda_{min}$ be respectively the biggest and the smallest eigenvalues of the martensite deformation matrices $\mt U_i\in K$. Let also $\vc y\in W^{1,\infty}(\Omega;\R^3)$ satisfy \eqref{h1}, \eqref{h2}. Then, we have $$|\mathfrak{D}-1|\leq |\vc a|\leq\lambda_{max}-\lambda_{min},\qquad \text{a.e. in } \Omega.$$
\end{proposition}
\begin{proof}
The first inequality follows trivially from Cauchy-Schwarz and the fact that $\vc a\cdot \vc n = \mathfrak{D}-1$. In order to get the other one, we observe that by the polar decomposition theorem we have that 
$$
\nabla \vc y(\vc x) = \mt R(\vc x)\mt F(\vc x),
$$
for almost every $\vc x\in\Omega$, where $\mt R(\vc x)\in SO(3)$ and $\mt F(\vc x)$ is symmetric positive definite. The argument below holds for almost every $\vc x\in\Omega$. By arguing as in \cite{BallCarstensen2} one can deduce that, in order to have a rank-one connection with the identity matrix, the eigenvalues of $\mt F$, namely $\sigma_{min}\leq\sigma_{mid}\leq\sigma_{max}$, must satisfy
$$
\sigma_{mid}=1,\qquad\qquad\sigma_{min}\sigma_{max}=\mathfrak{D}.
$$
Therefore, we have
\begin{equation}
\label{stima 2 norma}
\det \mt F = \mathfrak{D} = \sigma_{min}\sigma_{max},\qquad\qquad
\tr(\mt F^T\mt F) = 1+\sigma_{min}^2+\sigma_{max}^2.
\end{equation}
On the other hand,
$$
\nabla \vc y^T\nabla \vc y = \mt F^T\mt F = \mt 1 +\vc  a\otimes \vc n + \vc n\otimes \vc a +|\vc a|^2\vc n\otimes \vc n,
$$
which yields
\begin{equation}
\label{stima 1 norma}
\tr(\mt F^T\mt F) = 3 +2\vc a\cdot \vc n +|\vc a|^2 = 1+2\mathfrak{D}+|\vc a|^2.
\end{equation}
Therefore, by putting together~\eqref{stima 1 norma} and \eqref{stima 2 norma} we deduce
\[
0\leq |\vc a|^2 = (\sigma_{max}-\sigma_{min})^2\leq(\lambda_{max}-\lambda_{min})^2.
\]
Here we also made use of the following relation between eigenvalues proved in~\cite{BallCarstensen2}:
$$
\lambda_{min}\leq \sigma_{min}\leq1\leq \sigma_{max}\leq\lambda_{max}.
$$
\end{proof}

Another interesting property regards the divergence of $\vc n\otimes \vc a$: 
%{\color{red} Notation?|?|?|?  SBV, jump set, Gradient Young Measures}
\begin{proposition}
\label{divnulla}
Let $\vc z\in W^{1,\infty}(\Omega;\mathbb{R}^3)$ be such that 
$$
\nabla \vc z(\vc x) = \vc a(\vc x)\otimes \vc n(\vc x),\qquad \vc a\cdot \vc n=\mathfrak{D}-1\in \R,\qquad\text{a.e. in }\Omega.
$$
Then, $\nabla\cdot(\vc n\otimes \vc a)=\vc 0$ in the sense of distributions. 
Furthermore, if $\vc n\in W^{1,1}(\Omega;\mathbb{R}^3)$ with $|\vc n|=1$ a.e. in $\Omega$, then $\nabla\cdot \vc a=0$ in the sense of distributions.
\end{proposition}
\begin{proof}
On the one hand, we have
$$
\nabla(\nabla\cdot \vc z) = \nabla(\mathfrak{D}-1)=\vc0.
$$
On the other hand
$$
\nabla(\nabla\cdot\vc  z) = \nabla\cdot (\nabla \vc z)^T = \nabla\cdot (\vc n\otimes\vc  a).
$$
Here, both identities should be understood in the sense of distributions. By putting them together we hence get the first statement. As a consequence, if $\vc n\in W^{1,1}(\Omega;\mathbb{R}^3)$ such that $|\vc n|=1$ a.e. in $\Omega$, we can write,
\[
\begin{split}
\mint \vc a\cdot\nabla\ff\,\mathrm d\vc x &=  \mint \vc n^T \vc n\otimes \vc a \cdot\nabla\ff\,\mathrm d \vc x 
\\
& =\mint \vc n\otimes \vc a :\nabla(\vc n\ff)\,\mathrm d \vc x - \mint \ff \vc n\otimes \vc a :\nabla \vc n\,\mathrm d \vc x \\
&=- \frac12\mint \ff \vc a \cdot \nabla |\vc n|^2\,\mathrm d \vc x = 0,%\qquad \forall\ff\in C^\infty_c(\Omega),
\end{split}
\]
for every $\ff\in C^\infty_c(\Omega)$. %Here we also used the fact that $n$ is unitary.
\end{proof}

In general, it is not true that $\nabla\cdot \vc a =0$. Indeed, let us fix $\vc e\in\R^3$ with $|\vc e|=1$, and consider $\vc z$ to be of the form $\vc z(\vc x)=\vc e(\vc x\cdot \vc e)$. Let us also fix $c\in\R$ such that $\vc x\cdot \vc e = c$ for some $\vc x\in\Omega$ and define
$$
\vc a = \vc n=
\begin{cases}
\vc e,\qquad &\vc x\cdot \vc e < c,\\
-\vc e,\qquad &\vc x\cdot \vc e > c;
\end{cases}
%\qquad\qquad
%\vc n = 
%\begin{cases}
%\vc e,\qquad &\vc x\cdot \vc e < c,\\
%-\vc e,\qquad &\vc x\cdot \vc e > c.
%\end{cases}
$$
In this case, clearly $\vc a\cdot \vc n = 1$ a.e. in $\Omega$, and $\nabla\cdot (\vc n\otimes \vc a)=0$ in the distributional sense. However, $\vc a\in BV(\Omega;\R^3)$ satisfies
$$
\nabla \vc a = -2 \vc e\otimes \vc e \mathscr H^2 \mres %\Big\llcorner
\{\vc x\colon \vc x\cdot \vc e = c\},
$$
and, as $|\vc e|=1$ by hypothesis, $\nabla \cdot \vc a \neq 0$ in the distributional sense.
%\begin{figure}[!htbp]
%      \centering%  
%		{\label{counterex}\includegraphics[scale=0.7]{counterex.png}}
%\caption{
%\label{counterex}
%Picture of two parallel interfaces moving in opposite directions and meeting at a certain interface where $\vc a$, $\vc n$ are discontinuous. In this case, $\nabla\cdot(\vc n\otimes \vc a)=\vc 0$, but $\nabla\cdot \vc a \neq 0$. 
%}
%\end{figure}

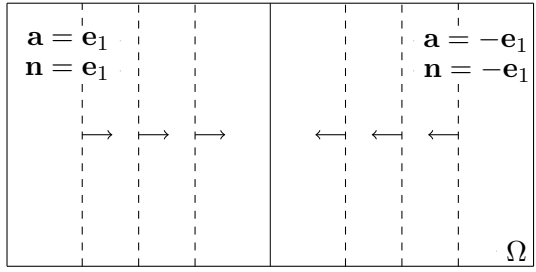
\begin{figure}
\centering
\begin{tikzpicture}
\draw[thin] (-3.5,3.5) -- (3.5,3.5);
\draw[thin] (3.5,3.5) -- (3.5,0);
\draw[thin] (3.5,0) -- (-3.5,0);
\draw[thin] (-3.5,0) -- (-3.5,3.5);
\draw[ultra thin] (0,0) -- (0,3.5);

\draw[dashed] (-2.5,0) -- (-2.5,2.5);
\draw[dashed] (-2.5,3.2) -- (-2.5,3.5);

%\draw[dashed] (-2,0) -- (-2,3.5);
\draw[dashed] (-1.75,0) -- (-1.75,3.5);
\draw[dashed] (-1,0) -- (-1,3.5);
%\draw[dashed] (-0.5,0) -- (-0.5,3.5);
\draw[dashed] (2.5,0) -- (2.5,2.5);
\draw[dashed] (2.5,3.2) -- (2.5,3.5);

%\draw[dashed] (-2,0) -- (-2,3.5);
\draw[dashed] (1.75,0) -- (1.75,3.5);
\draw[dashed] (1,0) -- (1,3.5);
\draw [->] (-2.5,1.75) -- (-2.1,1.75); 
\draw [->] (-1.75,1.75) -- (-1.35,1.75); 
\draw [->] (-1,1.75) -- (-0.6,1.75); 
\draw [->] (2.5,1.75) -- (2.1,1.75); 
\draw [->] (1.75,1.75) -- (1.35,1.75); 
\draw [->] (1,1.75) -- (0.6,1.75); 

\filldraw [red] (3.,0.2) circle (0pt) node[anchor=west,black] {$\Omega$};

\filldraw [red] (-2,3) circle (0pt) node[anchor=east,black] {$\vc a = \vc e_1$};
\filldraw [red] (-2,2.6) circle (0pt) node[anchor=east,black] {$\vc n = \vc e_1$};
\filldraw [red] (1.9,3) circle (0pt) node[anchor=west,black] {$\vc a = -\vc e_1$};
\filldraw [red] (1.9,2.6) circle (0pt) node[anchor=west,black] {$\vc n = -\vc e_1$};
\end{tikzpicture}
\caption{
\label{counterex}
Picture of two parallel interfaces moving in opposite directions and meeting at a certain interface where $\vc a$, $\vc n$ are discontinuous. In this case, $\nabla\cdot(\vc n\otimes \vc a)=\vc 0$, but $\nabla\cdot \vc a \neq 0$. 
}
\end{figure}

Keeping in mind the counterexample above, we extend the validity of the identity $\nabla\cdot \vc a=0$ in the following Corollary \ref{cor diva}. In this result we use the space of special functions with bounded variation on $\Omega$, namely $SBV(\Omega)$. If $\ff\in SBV(\Omega)$, its gradient is the sum of two Radon measures, one absolutely continuous with respect to the Lebesgue measure $\mathscr L^3$, and the other concentrated on a $2$-rectifiable set $S_\ff$, usually called the jump set. We denote by $\ff^-,\ff^+$ the trace of $\ff$ on the two sides of the jump set $S_\ff$ respectively. We refer the interested reader to \cite{AFP} and \cite{EG} for more details on this space.  
\begin{corollary} 
\label{cor diva}
Let $\vc z\in W^{1,\infty}(\Omega;\mathbb{R}^3)$ be such that 
\beq
\label{hp diva0 star}
\nabla \vc z(\vc x) = \vc a(\vc x)\otimes \vc n(\vc x),\qquad \vc a\cdot \vc n=\mathfrak{D}-1\in \R,\qquad\text{a.e. in }\Omega.
\eeq
Let $\vc a,\vc n\in SBV(\Omega;\mathbb{R}^3)\cap L^\infty(\Omega;\mathbb{R}^3)$, $|\vc n|=1$ a.e. in $\Omega$ and let 
%\begin{align}
%\label{hyp cor}
%\vc n^-(\vc x)\neq - \vc n^+(\vc x),\qquad \mathscr H^2\text{-a.e. }\vc x\in S_a\cap S_n\qquad \text{if $\mathfrak{D}\neq 1$},\\
%\label{hyp cor 2}
%\vc a^-(\vc x)= - \vc a^+(\vc x) \Rightarrow \vc n^-(\vc x)\neq - \vc n^+(\vc x),\qquad \mathscr H^2\text{-a.e. }\vc x\in S_a\cap S_n\qquad \text{if $\mathfrak{D} = 1$}.
%\end{align}
\begin{align}
\label{hyp cor}
\vc n^-(\vc x)\neq - \vc n^+(\vc x),\qquad &\text{if $\mathfrak{D}\neq 1$},\\
\label{hyp cor 2}
\vc a^-(\vc x)= - \vc a^+(\vc x) \Rightarrow \vc n^-(\vc x)\neq - \vc n^+(\vc x),&\qquad \text{if $\mathfrak{D} = 1$},
\end{align}
{for $\mathscr H^2-$a.e. $\vc x\in S_{\vc n}\cap  S_{\vc a}$.}
Then, $\nabla\cdot \vc a=0$ in the sense of distributions.
\end{corollary}

The following lemma is needed for the proof of Corollary \ref{cor diva}:
\begin{lemma}
\label{Had 1}
Let $\boldsymbol{\ff}\in H^1(\Omega,\R^3)$, with $\nabla\boldsymbol{\ff}\in BV(\Omega;\R^{3\times3})$. Then, there exists $\vc b(\vc x)\in L^1(S_{\nabla\boldsymbol\ff};\R^3)$ such that
$$
(\nabla\boldsymbol{\ff})^+(\vc x)-(\nabla\boldsymbol{\ff})^-(\vc x)=\vc b(\vc x)\otimes \vc m(\vc x),\qquad \mathscr H^2\text{-a.e. }\vc x\in S_{\nabla\boldsymbol\ff},
$$
with $\vc m(\vc x)$ being the normal to $S_{\nabla\boldsymbol\ff}$ in $\vc x$.
\end{lemma}

\begin{proof}
The proof of this type of result is usually done via a blow up argument and exploits continuity. This may be possible here, but we use a slightly different proof. Let $\boldsymbol\ff \in H^1(\Omega;\R^3)$ and let $A$ be a Lipschitz open subset of $\Omega$. Since $\nabla\times\nabla\boldsymbol\ff=0$ in the sense of distribution, we have that the following integration by parts formula holds (see \cite[Ch. 2, (2.18)]{GR})
\beq
\label{integr by part}
\int_A \nabla\ff_i\cdot\nabla\times \boldsymbol\psi\,\mathrm d \vc x = \int_{\partial A} (\nabla\ff_i\times \vc m)\cdot\boldsymbol\psi\,\mathrm d \mathscr H^2,\qquad \forall \boldsymbol\psi\in H^1(\Omega;\R^3),
\eeq
with $\vc m$ being the outpointing normal to $A$. We remark that, as stated in \cite[Ch. 2, Thm 2.5]{GR}, $(\nabla\ff_i\times \vc m)$ is a well defined object in $H^{-\frac12}(\partial A)$, and the integral on the right hand side should be interpreted as $\langle \nabla\ff_i\times \vc m,\boldsymbol \psi\rangle _{H^{-\frac12},H^{\frac12}}$. 
Let now $S$ be a Lipschitz $2$-graph contained in $\Omega$ with normal $\vc m$ and let $U_i$ be a countable set of open neighbourhoods such that $U_i\subset \Omega$, $U_i\setminus S$ has exactly two connected components, namely $U^+_i$ and $U^-_i$, and $$ \mathscr H^2(S\setminus \bigcup U_i)=0.$$
We now define $(\nabla\ff_i\times \vc m)^\pm$ to be the objects of $H^{-\frac12}(S)$ satisfying \eqref{integr by part} respectively for $A=U_i^\pm$. %After recalling from \cite[Ch. 2, Thm 2.5]{GR} that these are well defined, 
From \eqref{integr by part} we deduce
\[
\begin{split}
0=\int_{U_i} \nabla\ff_i\cdot\nabla\times \boldsymbol\psi \,\mathrm d \vc x = \int_{U_i^+} \nabla\ff_i\cdot\nabla\times \boldsymbol\psi\,\mathrm d \vc x+\int_{U_i^-} \nabla\ff_i\cdot\nabla\times \boldsymbol\psi\,\mathrm d \vc x\\= \int_{S\cap U_i} \bigl((\nabla\ff_i\times \vc m)^+ - (\nabla\ff_i\times \vc m)^-\bigr)\cdot\boldsymbol\psi\,\mathrm d \mathscr H^2,%\qquad \forall \psi\in C^\infty_c(U_i;\R^3),
\end{split}
\]
for every $\boldsymbol\psi\in C^\infty_c(U_i;\R^3)$. By repeating this argument on all $U_i$, this implies
\beq
\label{Had weak}
\|(\nabla\ff_i\times \vc m)^+ - (\nabla\ff_i\times \vc m)^-\|_{H^{-\frac12}(S)}=0,%\qquad \mathscr H^2-\text{a.e. $\vc x\in S$,   }
\qquad i=1,2,3,
\eeq
which is a weak Hadamard jump condition for $H^1(\Omega;\R^3)$ on Lipschitz $2$-graphs. Furthermore, if $\nabla\boldsymbol\ff \in SBV(\Omega)$, we know that the set where it is discontinuous, namely $S_{\nabla\boldsymbol\ff}$ is $2$-rectifiable (see e.g., \cite{AFP}). Therefore, by Proposition \eqref{SIPUO} we can cover $S_{\nabla\boldsymbol\ff}$ with countably many Lipschitz graphs where \eqref{Had weak} holds. On the other hand, from \cite[Thm 3.77]{AFP} we know that the trace of $\nabla\ff_i$ is well defined for almost every point of $S_{\nabla\boldsymbol\ff}$. Collecting these two facts we thus deduce the desired result.

\end{proof}
\begin{remark}
\label{HJC_weak}
\normalfont
Equation \eqref{Had weak} is a very weak version of the Hadamard jump condition on Lipschitz surfaces $\Gamma$ with normal $ \vc m$ for functions $ \vc y\in H^1(\Omega;\R^3)$. Indeed, we can only make sense to the tangential trace $\nabla  \vc y\times \vc  m$ of $ \vc y$ on $\Gamma$ as an object of $H^{-\frac12}(\Gamma)$, and \eqref{Had weak} states that, across $\Gamma$, $\nabla  \vc y\times  \vc m$ must not jump as an object of $H^{-\frac12}(\Gamma)$, which is kind of an average sense.
\end{remark}

\begin{proof}
{It follows from the definition of jump points of a $BV$ function (see e.g., \cite{AFP,EG}) that \eqref{hp diva0 star} and $|\vc n|=1$ hold $\mathscr H^2$--almost everywhere on $S_\vc a\cup S_\vc n$. Therefore, under our hypotheses \eqref{hyp cor}--\eqref{hyp cor 2} $\vc a\otimes \vc n \in SBV(\Omega)\cap L^\infty(\Omega)$ and $S_{\vc a\otimes \vc n}= S_\vc a\cup S_\vc n$ up to an $\mathscr H^2$--negligible set. Here, $\vc a,\vc n$ are chosen to be the precise representatives for $\vc a,\vc n.$ Furthermore, Lemma \ref{Had 1} implies that a Hadamard jump condition must hold across $S_{\vc a\otimes \vc n}$, so that
\beq
\label{R1c}
\vc a^+\otimes \vc n^+ - \vc a^-\otimes \vc n^- = \vc b\otimes \vc m,\qquad \mathscr H^2\text{-a.e. on } S_{\vc a\otimes \vc n},
\eeq
for some $\vc b\in L^\infty(S_{\vc a\otimes \vc n};\R^3)$ and with $\vc m$ being the normal to $S_{\vc a\otimes \vc n}$. In case $\mathfrak{D}\neq 1$, this, together with \eqref{hyp cor} and $|\vc n|=1$, imply that the only possible scenarios are the following on $S_\vc a\cup S_{\vc n}$:
\begin{enumerate}[(a)]
\item\label{J1} $\vc n^+ = \vc n^- = \vc m$% on $S_a$
, in which case $\vc b=\vc a^+-\vc a^-$;
\item\label{J2} $\vc a^+ = \xi \vc a^-$% on $S_a$
, in which case $\vc m\parallel\xi \vc n^+-\vc n^-$ and $\vc b\parallel \vc a^+\parallel \vc a^-$,
%\item in case $\mathfrak{D}=1$, $n^+ = -n^- = m$ .% on $S_a$
%, in which case $m=\xi n^+-n^-$ and $b\parallel a^+\parallel a^-$.
\end{enumerate}
for some $\xi\in L^\infty(S_{\vc n})$. Taking the trace of \eqref{R1c} implies also $\vc b\cdot\vc  m =0$. As a consequence, by \eqref{J1}--\eqref{J2} $(\vc a^+-\vc a^-)\cdot \vc m =0$ $\mathscr H^2$-a.e. on $S_{\vc a}\cup S_{\vc n}$, that is, the divergence of $\vc a$ has no singular part. In case $\mathfrak{D}=1$, \eqref{hyp cor 2} allows also to have $\vc n^+ = -\vc n^- = \vc m$ in $S_\vc a$, when $\vc a^+\neq - \vc a^-$. In which case $(\vc a^+-\vc a^-)\cdot\vc  m =0$ follows just by the fact that $\vc a\cdot \vc n=0$.

 It just remains to check that the part of $\nabla\cdot \vc a$ which is absolutely continuous with respect to $\mathscr L^3$ is null as well. To this aim, we first observe that, by the chain rule for $BV$ functions (see e.g., \cite[Example 3.97]{AFP})
\beq
\label{proddivexp}
\vc 0 = \nabla\cdot(\vc n\otimes \vc a) = \bigl(\vc n(\nabla^a\cdot \vc a)+ \nabla^a \vc n \vc a \bigr)\mathscr L^3 +  \bigl(\vc n^+\otimes\vc a^+ - \vc n^-\otimes\vc a^- \bigr)\vc m \,\mathscr H^2\mres S_{\vc a\otimes\vc n},
\eeq
where we denoted by $\nabla^a$ the absolutely continous part of the gradient. Given \eqref{J1}--\eqref{J2} above, under our hypotheses we have $\bigl(\vc n^+\otimes\vc a^+ - \vc n^-\otimes\vc a^- \bigr)\vc m = \vc 0$ $\mathscr H^2-$a.e. on $S_{\vc a\otimes\vc n}$. Furthermore, as $|\vc n|=1$ a.e., we have
$$
\vc 0 = \nabla|\vc n|^2 = \vc n^T \nabla^a \vc n\, \mathscr L^3 .
$$
Therefore, multiplying \eqref{proddivexp} by $\vc n$ and exploiting the fact that $|\vc n|=1$ a.e., we thus get 
$$
\nabla^a\cdot \vc a = 0,\qquad \text{a.e. in }\Omega.
$$
Therefore, as $\vc a\in SBV(\Omega)$, for every $\phi\in C^1_c(\Omega)$ we have
\begin{align*}
-\mint \nabla\ff\cdot \vc a\, \mathrm d\vc x &= \int_{S_\vc a} \ff (\vc a^+-\vc a^-)\cdot\vc  m \,\mathrm d\mathscr H^2 + \int_{\Omega} \ff \nabla^a\cdot \vc a\,\mathrm d\vc x = 0%\\
%&= \int_{\Omega\setminus (S_\vc a\cup S_\vc n)} \ff \nabla\cdot \vc a\,\mathrm d\vc x=0,\qquad\forall \ff \in C^\infty_c(\Omega),
\end{align*}
which concludes the proof.}
\end{proof}

\noindent
We now focus on compound twins. Thanks to Proposition \ref{comp dom char} we know that if two martensite variants $\mt U_1,\mt U_2\in\R^{3\times3}_{Sym^+}$ form a compound twin, then there exists $\mu>0$, $\vc v\in\mathbb{S}^2$ such that $\mt U_1\vc v = \mt U_2\vc v = \mu \vc v$. In this case, \cite[Theorem 2.5.1]{dolzman} states:
\begin{theorem}
Let $\mt U_1,\dots,\mt U_n\in \R^{3\times 3}_{Sym^+}$, such that $\det \mt U_i =\mathfrak{D}>0$, and such that $\mt U_i\vc v =\mu\vc v$ for some $\mu>0, \vc v\in \mathbb{S}^2$, for each $i=1,\dots,n$. Let also 
$$
H = \bigcup_{i=1}^n SO(3)\mt U_i.
$$
Then, there exists $l\in\mathbb{N}, \vc w_1,\dots,\vc w_l$ such that
\[
H^{qc}=
\Set{\mt F\in\R^{3\times 3} \,\Bigg|\; \text{\parbox{3.0in}{\centering 
$\det\mt F = \mathfrak{D}$,  $\mt F^T\mt F\vc v=\mu^2\vc v$ and 
$
\lvert\mt F\vc w_i\rvert^2 \leq \max_{j=1,\dots,n} \lvert\mt U_j\vc w_i\rvert^2
$ 
for each $i=1,\dots,l$.
}}}
\]
\end{theorem}
Therefore, in this simple case which includes compound twins, we can actually construct the quasiconvex hull of the set. An interesting result is given by the following lemma:
%
%some simple cases, it is possible to construct explicitly $K^{qc}$; in the next lemma we consider the case in which it is given by $SO(3)\mathcal{F}$ (see \cite{BallCarstensen1,BallJames2,dolzman}) where $\mathcal{F}\subset \R^{3\times3}_{Sym^+}$ is the set of matrices $\mt F$ such that 
%\beq
%\label{le Fs}
%\mt F^T\mt F = \left[\begin{array}{ ccc } \alpha & \gamma & 0 \\ \gamma & \beta & 0 \\ 0 & 0 & \mu^2 \end{array}\right]
%\eeq
%{%\color{red}
%and 
%\begin{equation}
%\label{constraints}
%\begin{cases}
%\alpha\beta-\gamma^2 &= \mathfrak{D}^2\mu^{-2}\\
%|\mt F\tilde{\vc  e}|^2 &\leq \max_{\mt U\in K}|\mt U\tilde{\vc e}|^2,\qquad \forall \tilde{\vc e} = (\vc e,0),\, \vc e\in\mathbb S^1,
%\end{cases}
%\end{equation}
%}where $\mathfrak{D}$ is the determinant and $\mu$ is an eigenvalue of the martensitic variants.\\
%
%We are interested in solving (H1) given that $\nabla \vc y \in K^{qc}$. This is equivalent to solve
%\begin{equation}
%\label{sistemone 0}
%(\mt 1+\vc n\otimes \vc a)(\mt 1+\vc a\otimes \vc n) = \mt F^T\mt F
%\end{equation}
% with $\mt F\in\mathcal F$. 
\begin{lemma}
\label{curved inter lemma}
Let $\mt U_1,\dots,\mt U_n\in \R^{3\times 3}_{Sym^+}$, such that $\det \mt U_i =\mathfrak{D}>0$, and such that $\mt U_i\vc v =\mu\vc v$ for some $\mu>0, \vc v\in \mathbb{S}^2$, for each $i=1,\dots,n$. Let also 
$$
H = \bigcup_{i=1}^n SO(3)\mt U_i,
$$
and $\mu\neq 1$. % and 
Then every map $\vc y\in W^{1,\infty}(\Omega;\R^3)$ satisfying $\nabla\vc y(\vc x)\in H^{qc}$ and $\nabla\vc y(\vc x) = \mt 1 + \vc a(\vc x)\otimes \vc n(\vc x)$ for a.e. $\vc x$ in $\Omega$ %{\color{red}if and only if}
%Suppose further \eqref{condizione 1} is satisfied, then 
is such that $\vc a\otimes\vc n$ is constant, and $|\vc a|=|\mu|^{-1}|\mathfrak{D}-\mu^2|$.
%and $\mu\neq 1$. 
%Then there exists $\vc y\in W^{1,\infty}(\Omega;\R^3)$ such that $\nabla\vc y(\vc x)\in K^{qc}$ and $\nabla\vc y(\vc x) = \mt 1 + \vc a(\vc x)\otimes \vc n(\vc x)$ for a.e. $\vc x$ in $\Omega$ {\color{red}if and only if}
%\beq
%\label{condizione 1}
%1\geq\frac{{\mu^2}(1-\mu^2)}{(\mathfrak{D}^2-\mu^4)}>0.
%\eeq
%Suppose further \eqref{condizione 1} is satisfied, then $\vc a\otimes\vc n$ is constant, and $|\vc a|=|\mu|^{-1}|\mathfrak{D}-\mu^2|$.
\end{lemma}
\begin{remark}
\label{condiz rem}
\rm
It comes up in the proof of Lemma \ref{curved inter lemma} that the following condition
\beq
\label{condizione 1}
1\geq\frac{{\mu^2}(1-\mu^2)}{(\mathfrak{D}^2-\mu^4)}>0
\eeq
is necessary in order to have the existence of $\vc a\in\R^3,\vc n\in\mathbb{S}^2$ such that $\mt 1+\vc a\otimes\vc n\in H^{qc}.$ We refer the reader to \cite{BallCarstensen2} for some stricter necessary conditions for this to hold.
\end{remark}
%%%%%%%%
%%%%%%%%%CASE MU = !??????
%%%%%%%%%
\begin{proof}
Thanks to a change of coordinates, we can suppose without loss of generality that $\vc v=\vc e_3$. In this case, every $\mt F\in H^{qc}$ satisfies 
\beq
\label{le Fs}
\mt F^T\mt F = \left[\begin{array}{ ccc } \alpha & \gamma & 0 \\ \gamma & \beta & 0 \\ 0 & 0 & \mu^2 \end{array}\right]
\eeq
and 
%{%\color{red}
%and 
\begin{equation}
\label{constraints}
\alpha\beta-\gamma^2 = \mathfrak{D}^2\mu^{-2}.
\end{equation}
We are first interested in solving $\mt 1+\vc a\otimes\vc n\in H^{qc}$ for $\vc a\in\mathbb{R}^3,\vc n\in \mathbb{S}^2$.
By \eqref{le Fs}, the first step is to solve the following nonlinear system of equations for the components of $\vc a$ and $\vc n$:
%% aggiungere la storia della rotazione che serve a portare le matrici con lo stesso auto valore a matrici con autovalore in comune e3
\begin{equation}
\label{sistemone}
\begin{cases}
1+ 2a_1n_1+|a|^2n_1^2 &= \alpha\\
1+ 2a_2n_2+|a|^2n_2^2 &= \beta\\
a_1n_2+a_2n_1+|a|^2n_1n_2 &= \gamma\\
a_3n_1+a_1n_3+|a|^2n_1n_3 &= 0\\
a_3n_2+a_2n_3+|a|^2n_2n_3 &= 0\\
1+ 2a_3n_3+|a|^2n_3^2 &= \mu^2\\
\end{cases}
\end{equation}
subject to the constraint \eqref{constraints}.
As a first step we compute $\alpha\beta-\gamma^2$ using system \eqref{sistemone}. After rearranging terms
$$
\alpha\beta-\gamma^2 = 2(a_1n_1+a_2n_2)+1+|\vc a|^2(n_1^2+n_2^2)-(a_1n_2-a_2n_1)^2.
$$
Since $|\vc n|=1$, using last equation of \eqref{sistemone} follows that
$$
|\vc a|^2(n_1^2+n_2^2)=|\vc a|^2(1-n_3^2)=|\vc a|^2+1+2a_3n_3-\mu^2,
$$
which leads to 
$$
\alpha\beta-\gamma^2 = 2(a\cdot n)+2+|\vc a|^2-\mu^2-(a_1n_2-a_2n_1)^2.
$$
Finally, by recalling~\eqref{H2 det} we deduce 
\begin{equation}
\label{cross 0}
\alpha\beta-\gamma^2 = 2\mathfrak{D}+|\vc a|^2-\mu^2-(a_1n_2-a_2n_1)^2.
\end{equation}
Thus, it is immediately seen that \eqref{cross 0} together with the first of \eqref{constraints} implies
\begin{equation}
\label{cross 1}
 (a_1n_2-a_2n_1)^2 = 2\mathfrak{D}+|\vc a|^2-\mu^2-\frac{\mathfrak{D}^2}{\mu^2} = |\vc a|^2-\frac{1}{\mu^2}(\mathfrak{D}-\mu^2)^2.
\end{equation}
On the other hand, exploiting the last equation of \eqref{sistemone} in the fourth and fifth one we have
\begin{align}
\label{cross 2}
a_2n_3-a_3n_2 &=-\frac{n_2}{n_3}(\mu^2-1)\\
\label{cross 2 bis}
a_3n_1-a_1n_3 &=\frac{n_1}{n_3}(\mu^2-1).
\end{align}
We note that it is legitimate to divide by $n_3$ since the last identity in \eqref{sistemone} together with $\mu\neq1$ implies $n_3\neq 0$. %Indeed the last equation of \eqref{sistemone} admits $n_3=0$ if and only if $\mu=1$. \\
By putting together \eqref{cross 1}-\eqref{cross 2 bis} we thus get
\beq
\label{cross}
 \vc a\times \vc  n = \left[\begin{array}{c} -\frac{n_2}{n_3}(\mu^2-1)\\ \frac{n_1}{n_3}(\mu^2-1)\\ \pm\sqrt{|\vc a|^2-\frac{1}{\mu^2}(\mathfrak{D}-\mu^2)^2} \end{array}\right].
\eeq
Since $(\vc a\times \vc n)\cdot\vc  n = 0$ we have
\beq
\label{cross 10}
n_3\sqrt{|\vc a|^2-\frac{1}{\mu^2}(\mathfrak{D}-\mu^2)^2}=0,
\eeq
which, as $n_3\neq0$, leads to $|\vc a|^2=\frac{1}{\mu^2}(\mathfrak{D}-\mu^2)^2$. In this case the norm of $\vc a$ is hence forced to be constant. We restrict ourselves without loss of generality to the case $\mathfrak{D}\neq\mu^2$.\\ %Furthermore, from \eqref{sistemone 0} and the fact that $\mu\neq 1$ follows that under the hypotheses of this Lemma, $\mathfrak{D}\neq \mu^2$.\\

We now want to write a in terms of the orthogonal vectors $(\vc n,\vc n^\perp,\vc n\times \vc n^\perp)$, where 
\beq
\label{n ampio}
\vc n^\perp := \left[\begin{array}{c}  n_2\\ -n_1\\ 0 \end{array}\right], \qquad \vc n\times \vc n^\perp =  \left[\begin{array}{c}  n_1n_3\\ n_2n_3\\ n_3^2-1 \end{array}\right].
\eeq
It is important to remark that, if $n_1=n_2=0$ then it is easy to see from \eqref{sistemone} that $a_1=a_2=0$, so we do not lose any generality with this representation.
%and that $\mathfrak{D}=\mu$ must hold. But in this case $a_3=\pm(\mathfrak{D}-1)$ and $n_3=\pm 1$ so that the thesis follows. Otherwise, w
We have,
$$
\vc a = a^{(1)}\vc n+a^{(2)}\vc n^\perp+a^{(3)}\vc n\times \vc n^\perp.
$$
As a first thing, recalling that $\det (\mt 1 +\vc a\otimes\vc n)=\mathfrak{D}$, we deduce that $a^{(1)}=\vc a\cdot\vc  n=\mathfrak{D}-1$. On the other hand, taking cross product of $\vc a$ with $\vc n$ follows that
\beq
\label{a cross}
 \vc n\times \vc a = a^{(2)}\vc n\times \vc n^\perp - a^{(3)}\vc n^\perp.
\eeq
%where we are using the fact that $u\times(v\times w) =(u\cdot w) v - (u\cdot v) w$ and $|n|=1$. 
A comparison of \eqref{a cross} with \eqref{cross} and \eqref{cross 10} thus leads to $a^{(3)}=\frac{1}{n_3}(1-\mu^2)$, and $a^{(2)}=0$, which implies
\beq
\label{a eq}
\vc a = (\mathfrak{D}-1)\vc n+\frac{1}{n_3}(1-\mu^2)\vc  n\times \vc n^\perp = (\mathfrak{D}-\mu^2)\vc n-\frac1{n_3}(1-\mu^2)\vc e_3,
\eeq
with $\vc e_3=(0,0,1)^T.$ 
As a first thing now we want to check whether it is possible to have $|\vc a|^2=\frac{1}{\mu^2}(\mathfrak{D}-\mu^2)^2$. Since the orthogonal vectors $\vc n$ and $\vc n\times \vc n^\perp$ satisfy $|\vc n|=1$ and $|\vc n\times \vc n^\perp|^2 = 1-n_3^2$, by rearranging the terms it is possible to obtain
\beq
\label{n3 eq}
n_3^2 = \frac{(1-\mu^2)^2}{(1-\mu^2)^2+\frac{1}{\mu^2}(\mu^2-\mathfrak{D})^2-(\mathfrak{D}-1)^2} = \frac{(1-\mu^2)}{\frac{1}{\mu^2}(\mathfrak{D}^2-\mu^4)}.
\eeq

It is easy to check that a pair of vectors $(\vc n,\vc a)$ with $\vc a$ defined in terms of $\vc n$ as in~\eqref{a eq}, and where $n_3$ is given by~\eqref{n3 eq}, satisfies the equations of~\eqref{sistemone}. Thus, it turns out that \eqref{condizione 1} in necessary in order not to contradict $|\vc n|=1$, $1-n_3^2=n_1^2+n_2^2\geq0$ and $n_3^2>0$ (see also Remark \ref{condiz rem}).\\% must hold, which imply thanks to \eqref{n3 eq} that \eqref{condizione 1} have to be satisfied.\\

\noindent
Let us now check when a map $\vc y\in W^{1,\infty}(\Omega;\R^3)$ such that $\nabla\vc y(\vc x)$ satisfies \eqref{le Fs}--\eqref{constraints} and $\nabla\vc y(\vc x) = \mt 1 + \vc a(\vc x)\otimes \vc n(\vc x)$ for a.e. $\vc x$ in $\Omega$. We have to verify that conditions expressed in \eqref{H3} hold, that is, we have to check when the constructed deformations are actually gradients, by verifying that $\nabla\times(a_i\vc n)=0$ for $i=1,2,3$, in the distributional sense. Since $\mathfrak{D}$ and $\mu$ are constants, so is $|n_3| = \sqrt{(\mu^2-\mu^4)(\mathfrak{D}^2-\mu^4)^{-1}}$ . Hence, by \eqref{a eq}, $|a_3|$ is constant and non zero as long as $\mu^2\neq\mathfrak{D}$ and $\mu\neq 1$, which is our case. Furthermore, $n_3=\sgn(n_3)|n_3|$ and $a_3=\pm |a_3| \sgn{(n_3)}$, where the sign depends on $\mu,\mathfrak D$ only and is hence fixed. We can hence suppose without loss of generality $a_3= |a_3| \sgn{(n_3)}$. Therefore,
$$\nabla\times(a_3\vc n)=|a_3|\nabla\times (\sgn(n_3)\vc n)=0,
$$
implies
$$
\curl{(\sgn(n_3)\vc n)}=0,
$$
in the sense of distributions. This implies the existence of $\psi\in W^{1,2}(\Omega)$ such that $\sgn(n_3)\vc n=\nabla\psi$ (see e.g., \cite{GR}). On the other hand, by Proposition \ref{divnulla}
$$\nabla\cdot(\vc a n_3)=|n_3|\nabla\cdot(\sgn(n_3) \vc a)=0,
$$
which implies
$$
\nabla\cdot (\sgn(n_3) \vc a)=0,
$$
in the sense of distributions. Furthermore,
we have from \eqref{a eq} that
$$
\nabla \cdot (\sgn(n_3) \vc a) = (\mathfrak{D}-\mu^2)\nabla\cdot (\sgn(n_3) \vc n) % \qquad i=1,2
$$
which thus implies that $\psi$ is harmonic in the sense of distributions. By using standard elliptic theory (see e.g., \cite{Evans}) we can thus deduce that $\psi\in C^\infty(\Omega)$. On the other hand, as $|\nabla\psi|^2=1$, we have
$$
\nabla\psi^T\nabla^2\psi=0, \qquad\text{for all $\vc x\in\Omega$},
$$
which implies that one eigenvalue of the Hessian matrix is null. Another eigenvalue is $0$ as $\vc e_3$ is a left eigenvector related to $0$ for $\nabla^2\psi$, and is not parallel to $\vc n$ unless $\vc n=\vc e_3$, in which case \eqref{n3 eq} forces $\sgn(n_3)\vc n=\vc e_3$ everywhere. The fact that $\psi$ is harmonic, i.e., $\tr \nabla^2\psi = 0$, thus means that $\nabla\psi$ is constant. Therefore, $\sgn(n_3)\vc n$ is constant, and as a consequence of \eqref{a eq}, so is $\sgn(n_3) \vc a$ and $\vc a\otimes \vc n$.
\end{proof}

{
\begin{remark}
\rm
In case $\mu=1$, it is not possible to deduce the same rigidity as in Lemma \ref{curved inter lemma}. Indeed, it can be deduced from equations \eqref{sistemone} that $\mu =1$ implies either $n_3=a_3=0$, or $\vc a = (\mathfrak{D}-1)\vc n$, but in the latter case the last equation in \eqref{sistemone} implies $\vc a\otimes\vc n = \mt 0$. The problem becomes thus two-dimensional, and we can rewrite 
$$
\vc n = \bigl( n_1,n_2,0\bigr)^T,\qquad \vc n^\perp = \bigl( n_2,-n_1,0\bigr)^T,\qquad \vc a(s) = (\mathfrak{D}-1)\vc n + s\vc n^\perp,
$$
for some $s\in \R, n_1,n_2\in [-1,1]$ with $n_1^2+n_2^2=1$. Now, take two martensite variants, for example, 
%$$\mt U_1= \left[\begin{array}{ ccc } \frac12(\lambda_m+\lambda_M) & \frac12(\lambda_m-\lambda_M) & 0 \\ \frac12(\lambda_m-\lambda_M) & \frac12(\lambda_m+\lambda_M) & 0 \\ 0 & 0 & 1 \end{array}\right],\qquad \mt U_2= \left[\begin{array}{ ccc } \frac12(\lambda_m+\lambda_M) & -\frac12(\lambda_m-\lambda_M) & 0 \\ -\frac12(\lambda_m-\lambda_M) & \frac12(\lambda_m+\lambda_M) & 0 \\ 0 & 0 & 1 \end{array}\right]$$ 
$$\mt U_1= \frac12\left[\begin{array}{ ccc } \lambda_m+\lambda_M & \lambda_m-\lambda_M & 0 \\ \lambda_m-\lambda_M & \lambda_m+\lambda_M & 0 \\ 0 & 0 & 2 \end{array}\right],\qquad \mt U_2= \frac12\left[\begin{array}{ ccc } \lambda_m+\lambda_M & \lambda_M-\lambda_m & 0 \\ \lambda_M-\lambda_m & \lambda_m+\lambda_M & 0 \\ 0 & 0 & 2 \end{array}\right],$$ 
generating a compound twin, and such that their biggest and smallest eigenvalue, namely $\lambda_M$ and $\lambda_m$, satisfy $\lambda_m<1<\lambda_M$. Assume further, for simplicity, $\lambda_m^2+\lambda_M^2>2$ and $\frac12\leq\lambda_m\lambda_M<1$. We can take for example $\lambda_m=0.9$ and $\lambda_M=1.1$. It can be computed that $(SO(3)\mt U_1\cup SO(3)\mt U_2)^{qc}$ coincides with the set of matrices $\mt F$ satisfying \eqref{le Fs}--\eqref{constraints}, that is such that $0<\alpha,\beta\leq \frac12(\lambda_m^2+\lambda_M^2)$ and $\alpha\beta\geq \lambda_m\lambda_M$. Consider now $\mt F =\mt 1 + \vc a(s)\otimes\vc n$, then
$$
\alpha = 1 + n_1^2(\mathfrak{D}^2-1+s^2) + s n_1n_2,\qquad\beta = 1 + n_2^2(\mathfrak{D}^2-1+s^2) - s n_1n_2.
$$
%
%
%Let $\vc n$ and $\hat{a}_0$ be such that $\mt 1+\vc a(\hat a_0)\otimes\vc n = \mt R(\mt U_1+\frac 12\vc b\otimes \vc m)$ for some $\mt R\in SO(3)$ and where $\vc b\in\R^3,\vc m\in \mathbb{S}^2$ are as in \eqref{compatib condit}. Then, 
Choosing for simplicity $n_2=0$ and $s$ small enough, we get $$0<\alpha,\beta< \frac12(\lambda_m^2+\lambda_M^2),\qquad\alpha\beta\geq \lambda_m\lambda_M=\mathfrak{D}^2.$$ Thus, there exists $\eps>0,$ and an open interval $[-\eps,\eps]$ such that for every smooth function $f\colon \R \to [-\eps,\eps]$ we have that
$$
\mt 1 + \vc a(f(\vc x\cdot \vc n^\perp))\otimes \vc n,
$$
is the gradient of a smooth map $\vc y$, which is not constant, and which satisfies $\nabla\vc y(\vc x) \in (SO(3)\mt U_1\cup SO(3)\mt U_2)^{qc}$ and \eqref{h1}--\eqref{h2} for each $\vc x$. Therefore, a rigidity result as the one in Lemma \ref{curved inter lemma} does not hold in general when $\mu=1$.
\end{remark}
\begin{remark}
\rm 
In \cite{FDP3} the author proved that for cubic to monoclinic II phase transitions (and hence also for its special cases of cubic to orthorhombic and cubic to tetragonal phase transitions) necessary and sufficient condition to satisfy (CC1)--(CC2) with a compound twin is to have $\mu =1$. Is therefore not surprising that the case $\mu = 1$ is a special case for Lemma \ref{curved inter lemma}. This is coherent also with Proposition \eqref{nuovorisultatone} below.
\end{remark}
}

{
By adding the further hypotheses that $\vc y|_{\partial\Omega}$ is the restriction on $\partial\Omega$ of a $1-1$ map, we can extend Lemma \ref{curved inter lemma} to general two well problems
\begin{proposition}
\label{nuovorisultatone}
Let $\mt U,\mt V\in\R^{3\times 3}_{Sym^+}$ and $\mt R_I,\mt R_{II}\in SO(3)$, $\vc b_I,\vc b_{II},\vc m_{I},\vc m_{II}\in\R^3\setminus\{\vc 0\}$ satisfy
$$
\mt R_i\mt V = \mt U +\vc b_i\otimes\vc m_i,\qquad i=I,II,
$$
where $(\mt U,\vc b_{I},\vc m_{I})$, $(\mt U,\vc b_{II},\vc m_{II})$ do not fulfil (CC2). Assume $\vc y\in W^{1,\infty}(\Omega;\mathbb{R}^3)$ is such that $\vc y|_{\partial\Omega} = \vc y_0|_{\partial\Omega}$ for some $\vc y_0\in C(\overline{\Omega};\R^3)$ which is $1-1$ in $\Omega$, and
\beq
\label{superipotesi}
\nabla\vc y(\vc x)\in \bigl(SO(3)\mt U \cup SO(3)\mt V\bigr)^{qc},\qquad \nabla\vc y(\vc x)=\mt 1+\vc a(\vc x)\otimes \vc n(\vc x),
\eeq
a.e. in $\Omega$, for some $\vc a\in L^\infty(\Omega;\mathbb{R}^3), \vc n \in L^\infty(\Omega;\mathbb{S}^2) $. Then,
$$
\nabla\vc y = \text{constant}.
$$
\end{proposition}
\begin{proof}
Following the approach devised in \cite{BallJames2}, we introduce the orthonormal system of coordinates
$$
\vc u_1^i:=\frac{\mt U^{-1}\vc m_i}{|\mt U^{-1}\vc m_i|} ,\qquad \vc u_3^i := \frac{\vc b_i}{|\vc b_i|}, \qquad\vc u_2^i := \vc u_3^i\times\vc u_1^i,
$$
with $i=I,II$ to be chosen later, and let 
$$
\mt L_i:= \mt U^{-1}\bigl( \mt 1 -\delta^i\vc u_3^i\otimes \vc u_1^i\bigr),\qquad \delta^i =\frac{1}{2}|\mt U^{-1}\vc m_i||\vc b_i|.
$$
We set $\vc z^i(\vc x):=\vc y(\mt L_i\vc x)$ and the problem becomes equivalent to find a map $\vc z^i\in W^{1,\infty}(\Omega;\R^3)$ such that
\beq
\label{defKqcT}
\nabla \vc z^i(\vc x) \in \bigl(SO(3)\mt S^-_i\cup SO(3)\mt S^+_i\bigr)^{qc},\qquad\text{a.e. $\vc x\in \Omega^{L_i}:=\bigl\{\vc x\colon \mt L_i\vc x\in \Omega	\bigr\}$},
\eeq
with $\mt S^\pm_i = \mt 1 \pm \delta^i\vc u_3^i\otimes\vc u_1^i,$ and
\beq
\label{bcz1}
\nabla\vc z^i(\vc x) = \mt L_i + \vc a(\mt L_i\vc x)\otimes \mt L_i^T\vc n(\mt L_i\vc x),\qquad\text{a.e. in $\Omega^{\mt L_i}$,}
\eeq
where $\vc a\in L^\infty(\Omega;\R^3)$, $\vc n\in L^\infty(\Omega;\mathbb{S}^2)$ are as in \eqref{superipotesi}. By \cite{BallJamesPlane}, $\vc z^i$ is a plane strain and satisfies
\beq
\label{bcz2}
\vc z^i(\vc x) = \mt Q\bigl(z_1^i(s_1^i,s_3^i)\vc u_1^i + s_2 \vc u_2^i + z_3^i(s_1^i,s_3^i) \vc u_3^i\bigr),
\eeq
for some $\mt Q\in SO(3),$ some Lipschitz functions $z_1^i,z_2^i$, and where $s_j^i = \vc x\cdot \vc u_j^i$, $j=1,2,3.$ As a consequence, from \eqref{bcz1}--\eqref{bcz2} we deduce
\begin{align*}
\vc u_2^i = \mt Q^T\nabla\vc z^i(\vc x) \vc u_2^i = \mt Q^T\mt L_i \vc u_2^i + \mt Q^T\vc a(\mt L_i\vc x)\bigl(\vc n(\mt L_i\vc x)\cdot \mt L_i\vc u_2^i\bigr),\\
\vc u_2^i = (\nabla\vc z^i(\vc x))^T\mt Q \vc u_2^i = \mt L_i^T\mt Q \vc u_2^i + \mt L_i^T\vc n(\mt L_i\vc x)\bigl(\vc a(\mt L_i\vc x)\cdot \mt Q\vc u_2^i\bigr), 
\end{align*}
a.e. in $\Omega^{\mt L_i}.$ That is,
\begin{align}
\label{aarigid1}
\vc a(\vc x)\bigl(\vc n(\vc x)\cdot \mt L_i\vc u_2^i\bigr) = (\mt Q-\mt L_i) \vc u_2^i,\\
\label{aarigid2}
\vc n(\vc x)\bigl(\vc a(\vc x)\cdot \mt Q\vc u_2^i\bigr) = (\mt L_i^{-T}-\mt Q) \vc u_2^i, 
\end{align}
a.e. in $\Omega.$ Let us now consider the function
$$
f_i(\mu) = \det \bigl( (\mt U + \mu\vc b_i\otimes\vc m_i)^T(\mt U + \mu\vc b_i\otimes\vc m_i) - \mt 1 \bigr),\qquad \mu \in [0,1],\, i=I,II.
$$
Thanks to \cite[Prop. 5]{BallJames1} we know that $f_i$ is a quadratic polynomial and $f_i(\mu)=f_i(1-\mu)$. We first notice that
\[
\begin{split}
f_i(\mu) &= 
\bigl(\det \mt U \bigr) \det\bigl((\mt U + \mu\vc b_i\otimes\vc m_i)	- (\mt U + \mu\vc b_i\otimes\vc m_i)^{-T}\bigr)
\\
&=
\bigl(\det \mt U^2 \bigr) \det\bigl((\mt 1-\mt U^{-2})	+ \mu 2\delta (\vc u_3^i \otimes \vc u_1^i + \vc u_1^i\otimes \mt U^{-2}\vc u_3^i)\bigr)
\\
&=
\det\bigl((\mt U^2-\mt 1)	+ \mu 2\delta (\vc u_3^i \otimes \mt U^2\vc u_1^i + \vc u_1^i\otimes \vc u_3^i)\bigr).
\end{split}
\]
A derivation of $f_i$ with respect to $\mu$ leads
\[
\begin{split}
&f_i'(\mu)\\
 &= 2\delta \bigl(\det \mt U^2 \bigr) \cof \bigl((\mt 1-\mt U^{-2})	+ \mu 2\delta (\vc u_3^i \otimes \vc u_1^i + \vc u_1^i\otimes \mt U^{-2}\vc u_3^i)\bigr) : (\vc u_3^i \otimes \vc u_1^i + \vc u_1^i\otimes \mt U^{-2}\vc u_3^i)\\
&= 
2\delta \bigl(\det \mt U^2 \bigr) \bigl( \cof (\mt 1-\mt U^{-2})\vc u_1^i\cdot (\vc u_3^i + \mt U^{-2}\vc u_3^i) + \mu 4\delta(\mt 1 - \mt U^{-2})\vc u_2^i\cdot( \vc u_1^i\times\mt U^{-2}\vc u_3^i) \bigr)\\
&= 
2\delta  \bigl( \cof (\mt U^2-\mt 1)\vc u_3^i\cdot (\vc u_1^i + \mt U^{2}\vc u_1^i) + \mu 4\delta(\mt U^2 - \mt 1)\vc u_2^i\cdot( \mt U^2\vc u_1^i\times \vc u_3^i) \bigr).
\end{split}
\]
Here we made use of the fact that $\cof\bigl(\sum_{i}\vc v_i\otimes \vc w_i\bigr) = \sum_{i< j} \vc v_i\times \vc v_j\otimes \vc w_i\times \vc w_j$. We now fix $i=I$ and claim that, under our hypotheses, there exist no $\mt Q\in SO(3)$ such that $(\mt Q-\mt L_I^{-T}) \vc u_2^I = \vc 0$. This, by \eqref{aarigid2} and the fact that $|\vc n|=1$ a.e. implies that $\vc n$ is, up to a change of sign, equal to a constant. That is, $\vc n\sgn(n_j)$ is constant, for some $j=1,2,3$ such that $|n_j|\neq 0$ a.e. in $\Omega$. To prove the claim we argue by contradiction, and notice that the existence of $\mt Q\in SO(3)$ satisfying $(\mt Q-\mt L_I^{-T}) \vc u_2^I = \vc 0$ implies $|\mt L_I^{-T}\vc u_2^I| = 1,$ that is
\beq
\label{magicaidentita}
(\mt U^2-\mt 1)\vc u_2^I\cdot\vc u_2^I=0.
\eeq
Let us notice now that
\beq
\label{paralleli1}
(\mt U^2\vc u_1^I\times\vc u_3^I)\times \vc u_2^I =\bigl(\mt U^2\vc u_1^I \cdot(\vc u_1^I\times\vc u_3^I) \bigr)\vc u_3^I.
\eeq
By making use of \eqref{10} in \eqref{paralleli1} we show that $\mt U^2\vc u_1^I \cdot(\vc u_1^I\times\vc u_3^I)=0,$ which by \eqref{paralleli1} implies that $\mt U^2\vc u_1^I\times\vc u_3^I$ is parallel to $\vc u_2^{I}$. Therefore, by \eqref{magicaidentita}, we get that $f_I'$ is constant in $\mu$. Furthermore, as $f_I$ is a quadratic polynomial and $f_I(\mu)=f_I(1-\mu)$ we have that $f_I'\Bigl(\frac12\Bigr)=0$ and hence $f_I'$ is identically $0.$ But, as 
$$
f_I'(\mu)|_{\mu=0} = 2\vc b_{I} \cdot\mt U\cof(\mt U^2-\mt 1) \vc m_{I} = 0,
$$ 
we contradict the assumption that $(\vc U,\vc b_{I},\vc m_{I})$ does not satisfy (CC2) concluding the proof of the claim. We now fix $i=II$ and claim that, under our hypotheses, there exist no $\mt Q\in SO(3)$ such that $(\mt Q-\mt L_{II}) \vc u_2^{II} = \vc 0$. This, by \eqref{aarigid1} and the fact that $\sgn(n_j)\vc n$ is a constant implies that $\sgn(n_j)\vc a$ is also a constant. To prove the claim we argue again by contradiction, and notice that the existence of $\mt Q\in SO(3)$ satisfying $(\mt Q-\mt L_{II}) \vc u_2^{II} = \vc 0$ implies $|\mt L_{II}\vc u_2^{II}| = 1,$ and thus
\beq
\label{magicaidentita2}
(\mt U^{-2}-\mt 1)\vc u_2^{II}\cdot\vc u_2^{II}=0.
\eeq
By making use of \eqref{11} we can now show that $\vc u_1^{II}\times\mt U^{-2}\vc u_3^{II}$ is parallel to $\vc u_2^{II}$, and hence, by \eqref{magicaidentita2}, that $f_{II}'$ is constant in $\mu$ and identically $0.$ But, noticing that 
$$
f_{II}'(\mu)|_{\mu=0} = 2\vc b_{II} \cdot\mt U\cof(\mt U^2-\mt 1) \vc m_{II} = 0,
$$ 
we contradict the assumption that $(\vc U,\vc b_{II},\vc m_{II})$ does not satisfy (CC2) concluding the proof of the second claim.

In conclusion, we proved that $\sgn(n_j)\vc n$ and $\sgn(n_j)\vc a$ are constants, and therefore so must be $\nabla\vc y$.
\end{proof}

\begin{remark}
It is clear from the proof of Proposition \ref{nuovorisultatone} that, if the type I solution $(\mt U,\vc b_I,\vc m_I)$ of the twinning equation \eqref{compatib condit} does not satisfy (CC2), but the type II solution $(\mt U,\vc b_{II},\vc m_{II})$ of \eqref{compatib condit} does, then we can guarantee that $\vc n$ in \eqref{superipotesi} is constant up to a change of sign. Similarly, if the type I solution $(\mt U,\vc b_I,\vc m_I)$ of the twinning equation \eqref{compatib condit} does satisfy (CC2), but the type II solution $(\mt U,\vc b_{II},\vc m_{II})$ of \eqref{compatib condit} does not, then the direction of $\vc a$ in \eqref{superipotesi} is constant. That is, there exists $\vc v\in\R^3$ such that $\vc a\times\vc v=0$ a.e. in $\Omega$. We refer the reader to Proposition \ref{type i inter} and Proposition \ref{type ii inter} for examples of non-affine maps when (CC2) is not satisfied.
\end{remark}
}

%%%%%%%%%%%%%%%%%%%%%%%%%%%%%%%%%%%%%%%%%
%%%%   Macroscopic curved interfaces in materials satisfying the cofactor conditions
%%%%%%%%%%%%%%%%%%%%%%%%%%%%%%%%%%%%%%%%%
\section{Macroscopic moving interfaces}
\label{results}
In this section, we use the theory of the previous sections to prove some results about moving interfaces in martensitic transformations. The results are different for different type of twins. We start with compound twins and we recall that, by Proposition \ref{comp dom char}, two martensite variants $\mt U_1,\mt U_2\in\R^{3\times3}_{Sym^+}$ form a compound twin if and only if there exist $\mu>0,\vc v\in \mathbb{S}^2$ such that $\mt U_1 \vc v= \mt U_2\vc v=\mu \vc v$. Thanks to Lemma \ref{curved inter lemma} we can prove that  in this case moving interfaces need to be planar and the related macroscopic gradient constant.

\begin{theorem}
\label{thm mio}
Let $\mt U_1,\mt U_2\in \R^{3\times3}_{Sym^+}$ be a compound twin and such that $$\mt U_1\vc w =\mt U_2\vc w = \mu\vc w$$ for some $\vc w\in\R^3$, $\mu\neq 1$. 
Then, every $\vc y$ satisfying the regular moving mask approximation and such that
$$
\nabla \vc y \in (SO(3)\mt U_1\cup SO(3)\mt U_2)^{qc},\qquad\text{a.e. in $\Omega$}$$% 
is constant, and the related moving interfaces planar.
%moving interface between austenite and the quasiconvex hull of $SO(3)U_1\cup SO(3)U_2$ satisfying (H1) and (H2) is planar and the deformation gradient $\nabla y$ is constant.
\end{theorem}
\begin{proof}
As $\vc y$ satisfies a regular moving mask approximation, by Theorem \ref{curved Hadamard} we know that $\nabla \vc y =\mt 1+\vc a\otimes \vc n$ for some $\vc a\in L^\infty(\Omega;\R^3),\vc n \in L^\infty(\Omega;\mathbb{S}^2)$. Since $\mu\neq1$, we can apply Lemma \ref{curved inter lemma} and thus deduce that $\nabla \vc y$ is constant in $\Omega.$ The function $\vc z(\vc x) := \vc y(\vc x)-\vc x$ is such that $\vc z\in W^{1,\infty}(\Omega;\R^3)$ and is constant in every connected component of $\Omega_A(t)$, for each $t\geq0$. Thus, $\Gamma(t)$ must be a (or at least the union of disconnected subsets of a) level-set for $\vc z$, and hence a plane (or union of disconnected planes) as $\nabla\vc z$ is constant and rank-one in $\Omega$. 
\end{proof}

By arguing in the same way, Theorem \ref{thm mio} can be generalised to a wider range of situations as stated in Theorem \ref{thm mio 2} below. This is relevant, for example, in the cubic to monoclinic transformation occurring in Zn\textsubscript{45}Au\textsubscript{30}Cu\textsubscript{25}, where there are 3 sets of four deformation gradients satisfying the hypotheses of Theorem \ref{thm mio 2}.\\

\begin{theorem}
\label{thm mio 2}
Let $\mt U_1,\dots,\mt U_N\in \R^{3\times3}_{Sym^+}$ be such that $\mt U_i\vc w = \mu\vc w$ and $\det \mt U_i=\mathfrak{D}$ for some $\vc w\in\R^3$, $\mu\neq 1$ and every $i=1,\dots,n$. %Let also $H = \bigcup_{i=1}^N SO(3)\mt U_i$. 
Then, every $\vc y$ satisfying the regular moving mask approximation and such that
$$
\nabla \vc y \in (\cup _{i=1}^NSO(3)\mt U_i)^{qc},\qquad\text{a.e. in $\Omega$}$$% 
is constant, and the related moving interfaces planar.
%moving interface between austenite and the quasiconvex hull of $SO(3)U_1\cup SO(3)U_2$ satisfying (H1) and (H2) is planar and the deformation gradient $\nabla y$ is constant.
\end{theorem}

{
An equivalent result can be proved, in the same way, for the general two well problem under the additional assumption that $\vc y$ coincides on $\partial\Omega$ with a $1-1$ map.
\begin{theorem}
\label{thm mio 3}
Let $\mt U,\mt V\in\R^{3\times 3}_{Sym^+}$ and $\mt R_I,\mt R_{II}\in SO(3)$, $\vc b_I,\vc b_{II},\vc m_{I},\vc m_{II}\in\R^3\setminus\{\vc 0\}$ satisfy
$$
\mt R_i\mt V = \mt U +\vc b_i\otimes\vc m_i,\qquad i=I,II,
$$
where $(\mt U,\vc b_{I},\vc m_{I})$, $(\mt U,\vc b_{II},\vc m_{II})$ do not fulfil (CC2). Then, every $\vc y$ satisfying the regular moving mask approximation, such that $\vc y|_{\partial\Omega} = \vc y_0|_{\partial\Omega}$ for some $\vc y_0\in C(\overline{\Omega};\R^3)$ which is $1-1$ in $\Omega$, and such that
$$
\nabla\vc y(\vc x)\in \bigl(SO(3)\mt U \cup SO(3)\mt V\bigr)^{qc},
$$
is constant, and the related moving interfaces planar.
%moving interface between austenite and the quasiconvex hull of $SO(3)U_1\cup SO(3)U_2$ satisfying (H1) and (H2) is planar and the deformation gradient $\nabla y$ is constant.
\end{theorem}
}
%\begin{remark}
%\normalfont
%This result can be easily extended to some more complex cases. For example in the case of cubic to monoclinic transformations as the one considered in \cite{JamesHyst} and in \cite{JamesNew}, there exist three sets of four martensitic deformation gradients satisfying the condition of \cite[Thm. 2.5.1]{dolzman}. By arguing as in the proof of Theorem~\ref{thm mio}, it is possible to use Lemma \ref{curved inter lemma} and deduce that %the deformation gradient 
%$\nabla \vc y$ generated by just these deformations gradients must be a constant.
%\end{remark}

The next result is related to type I twins satisfying the cofactor conditions. In this case we can prove that simple laminates can form macroscopically curved families of austenite-martensite interfaces with no transition layer. The proof strongly relies on Theorem \ref{CC for I}.%Indeed, we recall that Theorem \ref{CC for I} states that for each $\lambda\in[0,1]$
\begin{proposition}
\label{type i inter}
Let $\mt U_1$ and $\mt U_2$ be two martensitic variants and $\hat{\mt R}\in SO(3)$, $\vc b_I,\vc m_{I}\in\R^3$ be a type I solution to \eqref{compatib condit} and satisfying the cofactor conditions. Then, there exist $\mt R_0,\mt R_1\in SO(3)$, $\xi\in\R$, $\vc a_0\in \R^3$, $\vc n_0,\vc n_1\in\mathbb{S}^2$ such that
%~\eqref{forall lambda} with $\lambda \in L^\infty(\Omega)$, $\lambda\in[0,1]$ a.e. in $\Omega$ satisfies (H1) and (H2) if and only if $\nabla \lambda\times (\xi \vc n_1^{\omega^*}-\vc n_0^\omega)=\vc 0$ in the sense of distributions. If this conditions are satisfied, then $\vc y$ generates a family of non-planar austenite-martensite moving interfaces.
for every $\lambda \in L^\infty(\Omega;[0,1])$ satisfying $\nabla \lambda\times (\xi \vc n_1-\vc n_0)=\vc 0$ in the sense of distributions, there exists $\vc y\in W^{1,\infty}(\Omega;\R^3)$ with
$$
\nabla\vc y = \mt R_0 [(1-\lambda)\mt  U_1 +\lambda \hat{\mt R}\mt U_2]= \mt {1} + \vc a_0\otimes \bigr(\lambda \xi \vc n_1+(1-\lambda)\vc n_0\bigl),\qquad\text{a.e. in $\Omega$.}
$$ 
Furthermore, $\vc y$ satisfies a regular moving mask approximation, the related moving interfaces are curved, and $\nabla \cdot \bigl( |\lambda \xi \vc n_1+(1-\lambda)\vc n_0|\vc a_0 \bigr)=0$ in the sense of distributions.
\end{proposition}
\begin{proof}
From Theorem \ref{CC for I} and in particular from~\eqref{forall lambda} we know the existence of $\mt R_0,\mt R_1\in SO(3)$, $\xi\in\R$, $\vc a_0\in \R^3$, $\vc n_0,\vc n_1\in\mathbb{S}^2$ such that
$$
\mt R_0 [(1-\lambda)\mt  U_1 +\lambda \hat{\mt R}\mt U_2]= \mt {1} + \vc a_0\otimes \bigr(\lambda \xi \vc n_1+(1-\lambda)\vc n_0\bigl),\qquad\text{for all $\lambda\in[0,1]$.}
$$ 
%
%that $\mt R_0^\omega[\mt U_1+\lambda \vc b_I\otimes\vc m_I]$ is in the quasiconvex hull of $H_{2}:=SO(3)\vc U_1\cup SO(3)\vc U_2$ for every $\lambda \in [0,1]$. Furthermore equation~\eqref{forall lambda} is already of the form required by (H1). \\
We thus choose $\lambda\in L^\infty(\Omega)$ to be a function such that $\lambda\in[0,1]$ a.e., and define
\begin{align*}
&\vc a(\vc x):=\vc a_0 |\vc n_0+\lambda(\vc x)(\xi \vc n_1-\vc n_0)|,
&\vc n(\vc x):=\frac{\vc n_0+\lambda(\vc x)(\xi \vc n_1-\vc n_0)}{|\vc n_0+\lambda(\xi \vc n_1-\vc n_0)|},
\end{align*}
so that $\vc n$ has unitary norm. 
In the notation of Theorem \ref{CC for I}, we have
$$\det(\mt R_0[(1-\lambda) \mt U_1+\lambda \hat{\mt R}\mt U_2])=\det(\mt R_0)\det(\mt U_1+\lambda \vc b_I\otimes \vc m_I)=\det \mt U_1 
$$
where the Sherman-Morrison inversion formula and the fact that $\mt U_1^{-1}\vc n\cdot \vc a=0$ have been used. That is,
$$
\vc a_0\cdot (\vc n_0+\lambda(\xi \vc n_1-\vc n_0))
$$
is constant independently of $\lambda$, or, in an equivalent way,
\beq
\label{sono ortogonali}
\vc a_0\cdot (\xi \vc n_1-\vc n_0)=0. 
\eeq
We just need to check if it is possible to have
$$
\curl(a_i\vc  n)=0.
$$
Exploiting the definition of $\vc a$ and $\vc n$ get that this is satisfied if and only if
\beq
\label{perpend}
\nabla \lambda\times (\xi \vc n_1-\vc n_0)=\vc 0,
\eeq
in a weak sense. Therefore, if $\Omega$ is convex $\lambda$ must satisfy $\lambda(\vc x) = f(\vc x\cdot(\xi \vc n_1-\vc n_0) )$, for some $f\in L^\infty(\R;[0,1])$, and thus
$$
\vc y = \vc x+ \vc a_0 (\vc n_0\cdot \vc x+F(\vc x\cdot(\xi \vc n_1-\vc n_0))+\vc c,
$$
for some constant $\vc c\in\R^3$, and where $F(s)=\int_0^sf(s)\,\mathrm ds$. Therefore, after choosing 
$$I_T:=\bigl(\inf_{\vc x\in\Omega}G(\vc x),\sup_{\vc x\in\Omega}G(\vc x)\bigr),$$ 
where $G(\vc x)=\vc n_0\cdot \vc x+F(\vc x\cdot(\xi \vc n_1-\vc n_0))$, we deduce that the image of $\vc z(\vc x)=\vc y(\vc x)-\vc x$ is $\vc c+t\vc a_0$, for $t\in I_T$. If $\Omega$ is connected but not convex, then $\lambda$ might not be of the form $f(\vc x\cdot(\xi\vc n_1-\vc n_0))$, but the image of $\vc z(\vc x) = \vc y(\vc x) - \vc x$ is still contained in the one-dimensional line $\vc c+t\vc a_0$, for some $\vc c\in\R^3$ and $t$ in some bounded interval $\hat I_T$. 
We can thus use Corollary \ref{reg mov mask} and deduce the existence of a family of moving interfaces, which are also level sets for $\vc z(\vc x)$.\\
%by possibly dividing the domain $\Omega$ in subdomains, hypotheses of Corollary \ref{Coroll 1} hold, and the first statement follows from \eqref{blablabla} and Theorem \ref{curved Hadamard}.\\
In order to prove that $\nabla\cdot \vc a$, we first mollify $\lambda$ and, defined $\vc m$ as $\vc m:=\xi \vc n_1-\vc n_0$, notice that thanks to Fubini's theorem for distributions we can write
$$
\bigl\langle\nabla \lep \times \vc m, \boldsymbol\psi \bigr\rangle_{\mathcal D',\mathcal D} = \bigl\langle \nabla\lambda,(\vc m\times\boldsymbol\psi_\eps)\bigr\rangle_{\mathcal D',\mathcal D} = \bigl\langle \nabla\lambda\times \vc m,\boldsymbol\psi_\eps\bigr\rangle _{\mathcal D',\mathcal D}= 0,
$$
thanks to \eqref{perpend}, for all $\boldsymbol\psi\in C^\infty_c(\Omega,\R^3)$.  Therefore, $\nabla\lep \parallel \vc m$, and, by \eqref{sono ortogonali},
\beq
\label{tocitediva}
\nabla\lep\cdot \vc a_0=0.\eeq 
On the other hand, exploiting the smooth dependence on $\lambda$ of $\vc a$ and the fact that $\lep \to \lambda$ in $L^p(\Omega)$ for every $p\in[1,\infty)$, we have
$$
\mint \vc a(\lambda)\cdot\nabla\psi\,\mathrm d\vc x=\lim_{\eps\to 0}\mint \vc a(\lambda_\eps)\cdot\nabla\psi\,\mathrm d\vc x=-\lim_{\eps\to 0}\mint g'(\lep)\vc a_0\cdot\nabla\lep \psi\,\mathrm d\vc x
$$
for every $\psi\in C^\infty_c(\Omega)$ and where $g(s)= |\vc n_0+s (\xi \vc n_1-\vc n_0)|$. The last term in the chain of identities above is null due to \eqref{tocitediva}, and therefore the proof is concluded.
\end{proof}

Finally, a result related to type II twins satisfying the cofactor conditions:
\begin{proposition}
\label{type ii inter}
Let $\mt U_1$ and $\mt U_2$ be two martensitic variants and $\hat{\mt R}\in SO(3)$, $\vc b_{II},\vc m_{II}\in\R^3$ be a type II solution to \eqref{compatib condit} and satisfying the cofactor conditions. Then, there exist $\mt R_0\in SO(3)$, $\xi\in\R$, $\vc a_0,\vc a_1\in \R^3$, $\vc n_0\in\mathbb{S}^2$ such that
for every $\lambda \in L^\infty(\Omega;[0,1])$ satisfying $\nabla \lambda\times \vc n_0=\vc 0$ in the sense of distributions, there exists $\vc y\in W^{1,\infty}(\Omega;\R^3)$ with
$$
\nabla\vc y = \mt R_0 [(1-\lambda)\mt  U_1 +\lambda \hat{\mt R}\mt U_2]= \mt {1} + \bigr(\lambda \xi \vc a_1+(1-\lambda)\vc a_0\bigl)\otimes \vc n_0,\qquad\text{a.e. in $\Omega$.}
$$ 
Furthermore, $\vc y$ satisfies a regular moving mask approximation, and $$\nabla \cdot \bigl( \lambda \xi \vc a_1+(1-\lambda)\vc a_0 \bigr)=0$$ in the sense of distributions.
\end{proposition}
\begin{proof}
From Theorem~\eqref{CC for II} and in particular from~\eqref{forall lambda 2} we know the existence of
$\mt R_0,\mt R_1\in SO(3)$, $\xi\in\R$, $\vc a_0,\vc a_1\in \R^3$, $\vc n_0\in\mathbb{S}^2$ satisfying 
$$
\mt R_0 [(1-\lambda)\mt  U_1 +\lambda \hat{\mt R}\mt U_2]= \mt {1} + \bigr(\lambda \xi \vc a_1+(1-\lambda)\vc a_0\bigl)\otimes \vc n_0,\qquad\text{for all $\lambda\in[0,1]$.}
$$ 
Thus choose $\lambda\in L^\infty(\Omega)$ to be a function such that $\lambda\in[0,1]$ a.e., and define
\begin{align*}
&\vc a(\vc x):=(\vc a_0+\lambda(\vc x)(\xi \vc a_1-\vc a_0))
%&\vc n(\vc x):=\frac{\vc n_0^\omega}{|\vc n_0^\omega|}
\end{align*}
It is trivial to check, by arguing as in the proof of Proposition~\ref{type i inter}, that also (H2) holds.\\

Now, by taking the curl of $\nabla \vc y_i$ we deduce that
$$
\nabla \times(\lambda \vc n_0)=\nabla \lambda\times \vc n_0=0,
$$%\eeq 
in the sense of distributions. Therefore taking $\lambda$ such that $\nabla \lambda\parallel \vc n$ in a weak sense, by Proposition \ref{plane Hadamard} and Remark \ref{sono levelset i piani} follows the existence of a family of moving planar interfaces. The fact that $\nabla\cdot \vc a=0$ in the sense of distributions trivially follows from Proposition \ref{divnulla}.
\end{proof}

%$\mathfrak{d}\mathcal d$
\section{Experimental evidence}
The physical assumptions which were made in this work and some of the properties that have been deduced here are currently being investigated from an experimental perspective. Indeed, the authors of \cite{XC} have used X-ray Laue microdiffraction to measure the orientations and structural parameters of variants and phases in Zn\textsubscript{45}Au\textsubscript{30}Cu\textsubscript{25}. With this modern technique, the scanned area is meshed with small rectangles (e.g., in \cite{XC} authors use \SI{2}{\micro\metre} wide squares), and one can identify the phase and variant in each cuboid which has as a basis a rectangle of the mesh and depth of approximate \SI{2}{\micro\metre} from the sample surface. In cubes where a single phase or variant has been recognised, one can also measure the lattice parameters necessary to compute the average deformation gradient. In this way one can investigate what is happening at the phase interface by studying the lattice parameters in mesh rectangles where a martensite variant has been recognised and which have at least one neighbouring rectangle where the Laue microdiffraction was able to identify austenite.\\
\indent
In this way, the authors of \cite{XC} compute in some of the mesh cubes lying on the interface the number $\|\cof(\nabla \vc y -\mt{1})\|$, which, as it is easy to verify, is zero if and only if $\nabla\vc  y -\mt{1}$ is rank-one. Experimental results give $\|\cof(\nabla \vc y -\mt{1})\|$ to be of the order of $10^{-4}$, which seems small enough to be considered zero, and hence to justify \eqref{h1}. \\
\indent
Further investigations are ongoing to verify that $\nabla \vc y(\vc x)$ remains constant in time when $\vc x$ is not on the interface. This seems a reasonable assumption, as long as no external force acts on the sample and as long as one neglects other internal stresses giving rise to elastic deformations which, anyway, seem to be small compared to the deformations induced by the phase transition.\\
\indent
In conclusion, the data collected up till now seem to confirm the validity of the assumptions that we made in the present work. However, in the images in \cite{XC} there are many mesh cubes close to some of the phase interfaces where the X-ray Laue microdiffraction is not able to recognize any single variant or phase, and hence where the validity of the assumptions to get \eqref{h1} could be questioned or should be verified in some other way.\\
%Further work is being carried on by some of the authors of \cite{XC} to actually verify that macroscopic deformations satisfy $\nabla\cdot(n\otimes a) =0 $, as proved in Proposition \ref{divnulla}.

%
%$$
%\Omega\qquad\Omega_M(t)\qquad\Omega_A(t)\qquad\vc y (\Omega)\qquad\nabla \vc y \qquad \nabla \vc y= \mt R_1\mt U_1  \qquad \nabla \vc y= \mt R_2\mt U_2\qquad \Gamma(t)
%$$

\bibliographystyle{plain}
\bibliography{biblio}

%\begin{acknowledgements}
%\end{acknowledgements}


\end{document}